\documentclass{amsart}
\usepackage{amsthm,amsmath,amssymb,amsfonts}
\usepackage{lscape,graphics}
\usepackage{epsfig,pb-diagram}
\usepackage[hypertex]{hyperref}
\usepackage{verbatim}

\usepackage[margin=1.35in]{geometry}

\newtheorem{thm}{Theorem}
\newtheorem{lem}[thm]{Lemma}
\newtheorem{cor}[thm]{Corollary}
\newtheorem{prop}[thm]{Proposition}
\newtheorem*{statement}{Theorem}

\newtheorem{conj}[thm]{Conjecture}

\theoremstyle{definition}
\newtheorem{definition}[thm]{Definition}

\theoremstyle{remark}
\newtheorem{rem}{Remark}
\newtheorem{ex}[rem]{Example}

\numberwithin{rem}{section} \numberwithin{equation}{section}
\numberwithin{thm}{section}

\newcommand{\A}{{\mathbb A}}
\newcommand{\Ad}{\mathrm{Ad}}
\newcommand{\af}{\mathrm{af}}
\newcommand{\al}{\alpha}
\newcommand{\Aut}{\mathrm{Aut}}
\newcommand{\B}{\mathbb{B}}
\newcommand{\bb}{\mathfrak{b}}
\newcommand{\C}{\mathbb{C}}
\newcommand{\cd}{\mathrm{cd}}
\newcommand{\Coh}{\mathrm{Coh}}

\newcommand{\Frac}{\mathrm{Frac}}
\newcommand{\F}{\mathrm{Fun}}

\newcommand{\For}{\mathrm{For}}
\newcommand{\geh}{\mathfrak{g}}

\newcommand{\Gr}{\mathrm{Gr}}

\newcommand{\Hom}{\mathrm{Hom}}
\newcommand{\id}{\mathrm{id}}
\newcommand{\Img}{\mathrm{Im}}
\newcommand{\Inv}{\mathrm{Inv}}
\newcommand{\ip}[2]{\langle #1\,,\,#2\rangle}
\newcommand{\jd}{\varpi}
\newcommand{\jh}{j}

\renewcommand{\k}{\mathbf{k}}
\newcommand{\K}{{\mathbb K}}

\newcommand{\KKS}{g}
\newcommand{\la}{\lambda}
\newcommand{\La}{\Lambda}
\newcommand{\Lie}{\mathrm{Lie}}
\newcommand{\LL}{{\mathbb L}}
\newcommand{\nn}{\mathfrak{n}}
\newcommand{\OO}{\mathcal{O}}
\newcommand{\pKK}{\psi_{KK}}

\newcommand{\Po}{\mathbb{P}^1}

\newcommand{\pnt}{\mathrm{pt}}
\newcommand{\Q}{\mathbb{Q}}
\newcommand{\R}{\mathbb{R}}

\newcommand{\re}{\mathrm{re}}
\newcommand{\res}{\mathrm{res}}
\newcommand{\slh}{\widehat{\mathfrak{sl}}}

\newcommand{\Sym}{\mathrm{Sym}}

\newcommand{\Supp}{\mathrm{Supp}}

\newcommand{\ttt}{\mathfrak{t}}
\newcommand{\ve}{\varepsilon}
\newcommand{\Wz}{W_\af^I}

\newcommand{\Z}{{\mathbb Z}}

\begin{document}

\title{$K$-theory Schubert calculus of the affine Grassmannian}

\author{Thomas Lam}
\address{Department of Mathematics, Harvard University, Cambridge MA
02138 USA}
\email{tfylam@math.harvard.edu}

\author{Anne Schilling}
\address{Department of Mathematics, University of California, One
Shields Ave. Davis, CA 95616-8633 USA}
\email{anne@math.ucdavis.edu}

\author{Mark Shimozono}
\address{Department of Mathematics, Virginia Polytechnic Institute
and State University, Blacksburg, VA 24061-0123 USA}
\email{mshimo@vt.edu}

\date{April 2009}


\maketitle

\begin{abstract} We construct the Schubert basis of the torus-equivariant
$K$-homology of the affine Grassmannian of a simple algebraic group
$G$, using the $K$-theoretic NilHecke ring of Kostant and Kumar. This is the
$K$-theoretic analogue of a construction of Peterson in equivariant homology.

For the case $G= SL_n$, the $K$-homology of the affine Grassmannian
is identified with a sub-Hopf algebra of the ring of symmetric
functions.  The Schubert basis is represented by inhomogeneous
symmetric functions, called $K$-$k$-Schur functions, whose highest
degree term is a $k$-Schur function.  The dual basis in
$K$-cohomology is given by the affine stable Grothendieck
polynomials, verifying a conjecture of Lam.  In addition, we give a
Pieri rule in $K$-homology.

Many of our constructions have geometric interpretations using
Kashiwara's thick affine flag manifold.
\end{abstract}

\tableofcontents

\section{Introduction}
Let $G$ be a simple simply-connected complex algebraic group and $T
\subset G$ the maximal torus.  Let $\Gr_G$ denote the affine
Grassmannian of $G$.  The $T$-equivariant $K$-cohomology
$K^T(\Gr_G)$ and $K$-homology $K_T(\Gr_G)$ are equipped with
distinguished $K^T(\pnt)$-bases (denoted $\{[\OO_{X^I_w}]\}$
and $\{\xi_w\}$), called Schubert bases.  Our first main result is a
description of the $K$-homology $K_T(\Gr_G)$ as a subalgebra $\LL$
of the affine $K$-NilHecke algebra of Kostant and Kumar \cite{KK:K}.
This generalizes work of Peterson \cite{P} in homology.  Our second
main result is the identification, in the case $G = SL_n(\C)$, of
the Schubert bases of the non-equivariant $K$-(co)homology
$K_*(\Gr_G)$ and $K^*(\Gr_G)$ with explicit symmetric functions
called {\it $K$-$k$-Schur functions} and {\it affine stable
Grothendieck polynomials} \cite{Lam:affStan}. This generalizes work
of Lam \cite{Lam:Schub} in (co)homology.

\subsection{Kostant and Kumar's $K$-NilHecke ring}
Let $\geh$ be a Kac-Moody algebra and $X$ be the flag variety of
$\geh$.  Kostant and Kumar \cite{KK:K} studied the equivariant
$K$-theory $K^T(X)$ via a dual algebra $\K$ called the $K$-NilHecke
ring.  The ring $\K$ acts on $K^T(X)$ by Demazure divided
difference operators and scalar multiplication by
$K^T(\pnt)$.  In particular, they used $\K$ to define a
``basis'' $\{\psi_{KK}^v\}$ of $K^T(X)$ (elements of $K^T(X)$ are
{\it infinite} $K^T(\pnt)$-linear combinations of the ``basis'').


Kostant and Kumar use the ind-scheme $X_{\rm ind}$, which is an
inductive limit of finite-dimensional schemes.  Because of this,
classes in $K^T(X_{\rm ind})$ do not have an immediate geometric interpretation,
but are defined via duality in terms of geometric classes in
$K_T(X_{\rm ind})$.  We use instead the ``thick'' flag variety $X$ of
Kashiwara \cite{Kash}, which is an infinite-dimensional scheme.
This allows us to interpret the $K$-NilHecke ring operations
geometrically, and to describe (Theorem~\ref{T:psiloc}) the Schubert ``basis''
of $K^T(X)$, representing coherent sheaves $\OO_{X_w}$ of finite
codimensional Schubert varieties.  Our basis is different to that of
Kostant and Kumar.  On the other hand, in our treatment the
$K$-homology $K_T(X)$ is now defined via duality.

\subsection{The affine Grassmannian and the small torus GKM condition}
Let $\geh$ be a finite-dimensional simple Lie algebra, and
$\geh_\af$ the untwisted affine algebra.  Instead of using the
affine torus $T_\af$, we use the torus $T \subset G$ of the
finite-dimensional algebraic group, and study the equivariant
$K$-cohomology $K^T(X_\af)$ and $K^T(\Gr_G)$ of the affine flag
variety and affine Grassmannian.  We use the {\it affine
$K$-NilHecke ring for $\geh$}, still denoted $\K$, rather than the
slightly larger Kostant-Kumar $K$-NilHecke ring for $\geh_\af$.  The
corresponding affine NilHecke ring in cohomology was considered by
Peterson \cite{P}.

We describe (Theorem \ref{T:GKMPet}) the image of $K^T(X_\af)$ and
$K^T(\Gr_G)$ in $\prod_{w \in W_\af} K^T(\pnt)$ under localization
at the fixed points, where $W_\af$ denotes the affine Weyl group.
This is the $K$-theoretic analogue of a result of Goresky, Kottwitz, and MacPherson
\cite{GKM:2004} in homology.  We call the corresponding condition
the {\it small torus GKM condition}.  It is significantly more
complicated than the usual condition for GKM spaces \cite{GKM},
which would apply if we used the larger torus $T_\af$.  This
description gives an algebraic proof of the existence of a crucial
``wrong way'' map $K^T(X_\af) \to K^T(\Gr_G)$, which corresponds in
the topological category to $\Omega K \hookrightarrow LK \to
LK/T_{\R}$, where $K \subset G$ is a maximal compact subgroup,
$T_{\R}=T\cap K$, and $\Omega K$ and $LK$ denote the spaces of based and
unbased loops. The space of based loops $\Omega K$ is a topological model
for the affine Grassmannian~\cite{PS}.

Another description of the $K$-homology of the affine Grassmannian
is given by Bezrukavnikov, Finkelberg, and Mirkovi\'{c}~\cite{BFM},
though the methods there do not appear to be particularly suited to
the study of Schubert calculus.

\subsection{$K$-theoretic Peterson subalgebra and affine Fomin-Stanley subalgebra}
We let $\LL = Z_\K(R(T))$ denote the centralizer in $\K$ of the
scalars $R(T) = K^T(\pnt)$, and call it the {\it $K$-Peterson
subalgebra}. (This centralizer would be uninteresting if we had used
$T_\af$ instead of $T$.)  We generalize (Theorem \ref{T:kmap}) a
result of Peterson \cite{P} (see also \cite{Lam:Schub}) in homology:
\begin{statement}
There is a Hopf isomorphism $k: K_T(\Gr_G) \longrightarrow \LL$.
\end{statement}

The Hopf-structure of $K_T(\Gr_G)$ is derived from $\Omega K$.  We
also give a description (Theorem \ref{T:LL}) of the images
$k(\xi_w)$ of the Schubert bases under this isomorphism.

Next we consider a subalgebra $\LL_0 \subset \K_0$, called the {\it
$K$-affine Fomin-Stanley subalgebra}, of the affine 0-Hecke algebra.
We show that $\LL_0$ is the evaluation of $\LL$ at 0, and that it is
a model for the non-equivariant homology $K_*(\Gr_G)$.

\subsection{$G = SL_n$ and Grothendieck polynomials for the affine Grassmannian}
We now focus on $G = SL_n$.  In \cite{Lam:affStan}, the {\it affine
stable Grothendieck polynomials} $G_w(x)$ were introduced, where $w
\in W_\af$ is an affine permutation.  The symmetric functions
$G_w(x)$ lie in a completion $\hat{\La}^{(n)}$ of a quotient of the
ring of symmetric functions. A subset of the $\{G_w(x)\}$ form a
basis of $\hat{\La}^{(n)}$, and the dual basis $\{g_w(x)\}$, called
{\it $K$-theoretic $k$-Schur functions}, form a basis of a
subalgebra $\La_{(n)}$ of the ring of symmetric functions.

The symmetric functions $G_w(x)$ are $K$-theoretic analogues of the
affine Stanley symmetric functions in \cite{Lam:affStan}, and on the
other hand affine analogues of the stable Grothendieck polynomials
in \cite{FK, B}. The symmetric functions $g_w(x)$ are $K$-theoretic
analogues of the $k$-Schur functions $s_w(x)$ \cite{LLM, LM1, LM2},
and on the other hand the affine (or $k$-) analogues of the dual
stable Grothendieck polynomials \cite{L,LP}.

Using the technology of the $K$-affine Fomin-Stanley subalgebra, we
confirm a conjecture of Lam \cite{Lam:affStan}, by showing
(Theorem~\ref{T:symmfunc})

\begin{statement}
There are Hopf isomorphisms $K_*(\Gr_G) \cong \La_{(n)}$ and
$K^*(\Gr_G) \cong \hat{\La}^{(n)}$, identifying the homology
Schubert basis with the $K$-$k$-Schur functions $g_w(x)$, and the
cohomology Schubert basis with the affine stable Grothendieck
polynomials $G_w(x)$.
\end{statement}

This generalizes the main result of \cite{Lam:Schub}, and the
general idea of the proof is the same.

We also obtain a Pieri rule (Corollary \ref{C:Pieri}) for
$K_*(\Gr_G)$.  We give in Theorem \ref{T:ASG} a geometric
interpretation of $G_w(x)$ for any $w \in W_\af$ as a pullback of
a Schubert class from the affine flag variety to the affine
Grassmannian.  We conjecture that the symmetric functions
$G_w(x)$ and $g_w(x)$ satisfy many positivity properties
(Conjectures~\ref{C:g} and~\ref{C:G}).

\subsection{Related work}
Morse~\cite{M} gives a combinatorial definition of the affine stable
Grothendieck polynomials $G_w(x)$ in terms of affine set-valued tableaux
and also proves the Pieri rule for $G_w$.
The original $k$-Schur functions $s_w(x;t)$~\cite{LLM,LM1}, which arose in the study of Macdonald
polynomials, involve a parameter $t$.  It appears that a $t$-analogue
$g_w(x;t)$ of $g_w(x)$ exists , defined in a similar manner to
\cite[Conjecture 9.11]{LLMS}. The connection of $g_w(x;t)$ and Macdonald theory is
explored in~\cite{BM}.

Kashiwara and Shimozono~\cite{KS} constructed polynomials, called {\it affine Grothendieck
polynomials}, which represent Schubert classes in the $K$-theory of the affine flag manifold.  It is
unclear how affine Grothendieck polynomials compare with our symmetric functions.

\subsection*{Organization}
In Section~\ref{S:KK}, we review the constructions of the
$K$-NilHecke ring $\K$ and define our function ``basis''
$\{\psi^v\}$.
In Section~\ref{S:KM} we introduce Kashiwara's geometry of ``thick"
Kac-Moody flag varieties $X$ and the corresponding equivariant
$K$-cohomologies; we show how $\K$ corresponds to the geometry of
$X$.  Section~\ref{S:P} is devoted to equivariant $K$-theory for
affine flags and affine Grassmannians with the small torus $T$
acting using the level-zero action of the affine Weyl group.  In
Section~\ref{S:Palgebra} we introduce the affine $K$-NilHecke ring
and the $K$-Peterson subalgebra $\LL$, and prove that the latter is
isomorphic to $K_T(\Gr_G)$.  In Section~\ref{S:FS} we study the
$K$-affine Fomin-Stanley algebra.  In Section~\ref{S:G} we restrict
to $G = SL_n$ and describe the Hopf algebra isomorphisms between
$K_*(\Gr_G)$ and $K^*(\Gr_G)$ and symmetric functions explicitly.

A review of the cohomological NilHecke ring of Kostant-Kumar, the
affine NilHecke algebra $\A$, and some tables of the symmetric
functions $g_w$ and $G_w$ are provided in Appendix~\ref{S:A}.

\subsection*{Acknowledgements} We thank Masaki Kashiwara and Jennifer Morse for helpful
discussions and Bob Kottwitz for pointing us to~\cite{GKM:2004}.  Many thanks also to
Nicolas Thi\'ery, Florent Hivert, and Mike Hansen for their help with the open source
Mathematics Software Sage and $*$-Combinat~\cite{HT,Sage}.
This work was partially supported by the NSF grants DMS--0600677,
DMS--0501101, DMS--0652641, DMS--0652648, and DMS--0652652.

\section{Kostant-Kumar $K$-NilHecke ring}
\label{S:KK} One of the themes of \cite{KK:K} is that the Schubert
calculus of the torus-equivariant $K$-theory $K^T(X)$ of a Kac-Moody
flag manifold $X$, is encoded by the $K$-NilHecke ring $\K$, which
acts on $K^T(X)$ as {\it Demazure operators}.  We review the
constructions of \cite{KK:K} but use a different ``basis" of
$K^T(X)$, namely, the classes of equivariant structure sheaves of
finite codimensional Schubert varieties in the thick flag manifold
of \cite{Kash}.

For a statement $S$, we let $\chi(S)=1$ if $S$ is true and $0$ if
$S$ is false.

\subsection{Kac-Moody algebras}
Let $\geh$ be the Kac-Moody algebra over $\C$ associated with the
following data: a Dynkin node set $I$, symmetrizable generalized
Cartan matrix $(a_{ij})_{i,j\in I}$, free $\Z$-module $P$, linearly
independent simple roots $\{\al_i\mid i\in I\}\subset P$, the dual
lattice $P^*=\Hom_\Z(P,\Z)$, with simple coroots $\{\al_i^\vee\mid
i\in I\}\subset P^*$, such that $\ip{\al_i^\vee}{\al_j}=a_{ij}$
where $\ip{\cdot}{\cdot}:P^*\times P\to\Z$ is the pairing, with the
additional property that there exist fundamental weights
$\{\La_i\mid i\in I\}\subset P$ satisfying
$\ip{\al_i^\vee}{\La_j}=\delta_{ij}$.  Let $Q = \bigoplus_{i\in I}
\Z \al_i \subset P$ be the root lattice and $Q^\vee =
\bigoplus_{i\in I} \Z \al_i^\vee \subset P^*$ the coroot lattice.
Let $\geh = \nn_+\oplus \ttt\oplus \nn_-$ be the triangular
decomposition, with $\ttt \supset P^* \otimes_\Z \C$. Let $\Phi$ be
the set of roots and $\Phi^\pm$ the sets of positive and negative
roots, and let $\geh = \bigoplus_{\al\in\Phi} \geh_\al$ be the root
space decomposition. Let $W\subset\Aut(\ttt^*)$ be the Weyl group,
with involutive generators $r_i$ for $i\in I$ defined by
$r_i\cdot\la =\la-\ip{\al_i^\vee}{\la}\al_i$. For $i,j\in I$ with
$i\ne j$ let $m_{ij}$ be $2,3,4,6,\infty$ according as
$a_{ij}a_{ji}$ is $0, 1, 2, 3, \ge4$. Then $W$ has involutive
generators $\{r_i\mid i\in I\}$ which satisfy the braid relations
$(r_ir_j)^{m_{ij}} = \id$.  Let $\Phi^\re = \{w\al_i\mid w\in W,
i\in I\}\subset Q$ be the set of real roots and for $\al=w\al_i$ let
$r_\al = w r_i w^{-1}$ be the associated reflection and $\al^\vee =
w \al_i^\vee$ the associated coroot. Let $\Phi^{+\re} = \Phi^\re
\cap \Phi^+$ be the set of positive real roots.

\subsection{Rational form}
Let $T$ be the algebraic torus with character group $P$. The Weyl
group $W$ acts on $P$ and therefore on $R(T)$ and
$Q(T)=\Frac(R(T))$, where
$$R(T)\cong\Z[P]=\bigoplus_{\la\in P} \Z e^\la$$
is the Grothendieck group of the category of finite-dimensional
$T$-modules, and for $\la\in P$, $e^\la$ is the class of the
one-dimensional $T$-module with character $\la$.

Let $\K_{Q(T)}$ be the smash product of the group algebra $\Q[W]$
and $Q(T)$, defined by $\K_{Q(T)} = Q(T) \otimes_\Q \Q[W]$
with multiplication $$(q\otimes w)(p\otimes v) = q (w\cdot p) \otimes wv$$
for $p,q\in Q(T)$ and $v,w\in W$. We write $qw$ instead of $q\otimes w$.
For $i\in I$ define the Demazure operator \cite{Dem} $y_i\in \K_{Q(T)}$ by
\begin{equation*}
  y_i = (1-e^{-\al_i})^{-1} (1 - e^{-\al_i} r_i).
\end{equation*}
The $y_i$ are idempotent and satisfy the braid relations:
\begin{equation*}
 y_i^2 = y_i  \ \ \ \ \text{and}\ \ \ \
\underbrace{y_iy_j\dotsm}_{\text{$m_{ij}$ times}} =
\underbrace{y_jy_i\dotsm}_{\text{$m_{ij}$ times}}\; .
\end{equation*}
Define the elements $T_i\in \K_{Q(T)}$ by
\begin{equation} \label{E:Tdef}
  T_i = y_i - 1 = (1-e^{\al_i})^{-1} (r_i-1).
\end{equation}
We have
\begin{equation} \label{E:r}
  r_i = 1 + (1-e^{\al_i}) T_i.
\end{equation}
The $T_i$ satisfy
\begin{equation} \label{E:Tbraid}
T_i^2 = -T_i \ \ \ \ \text{and}\ \ \ \
\underbrace{T_iT_j\dotsm}_{\text{$m_{ij}$ times}} =
\underbrace{T_jT_i\dotsm}_{\text{$m_{ij}$ times}}\; .
\end{equation}
Let $T_w = T_{i_1} T_{i_2}\dotsm T_{i_N}\in \K_{Q(T)}$
where $w=r_{i_1}r_{i_2}\dotsm r_{i_N}$ is a reduced decomposition; it is well-defined
by \eqref{E:Tbraid}. It is easily verified that
\begin{equation*}
  T_i T_w = \begin{cases}
  T_{r_i w} &\text{if $r_iw>w$} \\
  -T_w &\text{if $r_iw<w$}
  \end{cases}\qquad \text{and} \qquad
  T_w T_i = \begin{cases}
  T_{wr_i} &\text{if $wr_i>w$} \\
  -T_w &\text{if $wr_i<w$}
  \end{cases}
\end{equation*}
where $<$ denotes the Bruhat order on $W$.
For $\al\in\Phi^{+\re}$ define $T_\al = (1-e^\al)^{-1}(r_\al-1)$.
 Let $w\in W$ and $i\in I$ be such that $\al=w\al_i$. Then
\begin{equation} \label{E:Talconj}
  T_\al = w T_i w^{-1}.
\end{equation}
$\K_{Q(T)}$ acts naturally on $Q(T)$.  In particular, one has
\begin{equation}\label{E:deriv}
  T_i \cdot (qq') = (T_i\cdot q) q' + (r_i\cdot q) T_i \cdot q'\qquad\text{for $q,q\in Q(T)$.}
\end{equation}
Therefore in $\K_{Q(T)}$ we have
\begin{equation}\label{E:Tcomm}
T_i \,q = (T_i\cdot q) + (r_i \cdot q) T_i\qquad\text{for $q\in Q(T)$.}
\end{equation}

\subsection{$0$-Hecke ring and integral form}
The $0$-Hecke ring $\K_0$ is the subring of $\K_{Q(T)}$ generated by
the $T_i$. It can also be defined by generators $\{T_i\mid i\in I\}$
and relations \eqref{E:Tbraid}. We have $\K_0 = \bigoplus_{w\in W}
\Z T_w$.

\begin{lem} \label{L:0Heckeacts} $\K_0$ acts on $R(T)$.
\end{lem}
\begin{proof} $\K_0$ acts on $Q(T)$, and $T_i$ preserves $R(T)$ by \eqref{E:deriv}
and the following formulae for $\la\in P$:
\begin{equation*}
  T_i \cdot e^\la = \begin{cases}
  e^\la(1+e^{\al_i}+\dotsm+e^{(\ip{\al_i^\vee}{\la}-1)\al_i}) & \text{if $\ip{\al_i^\vee}{\la} > 0$} \\
  0 & \text{if $\ip{\al_i^\vee}{\la}=0$}\\
-e^\la(1+e^{\al_i}+\dotsm+e^{(-\ip{\al_i^\vee}{\la}-1)\al_i}) & \text{if $\ip{\al_i^\vee}{\la} < 0$.}
  \end{cases}
\end{equation*}
\end{proof}

Define the {\it $K$-NilHecke ring} $\K$ to be the subring of
$\K_{Q(T)}$ generated by $\K_0$ and $R(T)$.  We have $\K_{Q(T)}
\cong Q(T) \otimes_{R(T)} \K$.  By \eqref{E:Tcomm}
\begin{equation} \label{E:Tbasis}
 \K = \bigoplus_{w\in W} R(T) T_w.
\end{equation}

\subsection{Duality and function ``basis"}
Let $\F(W,Q(T))$ be the right $Q(T)$-algebra of functions from $W$ to $Q(T)$
under pointwise multiplication and scalar multiplication $(\psi \cdot q)(w)=q \psi(w)$
for $q\in Q(T)$, $\psi\in \F(W,Q(T))$, and $w\in W$.
By linearity we identify $\F(W,Q(T))$ with left $Q(T)$-linear
maps $\K_{Q(T)} \to Q(T)$:
\begin{align*}
  \psi(\sum_{w\in W} a_w w) = \sum_{w\in W} a_w \psi(w).
\end{align*}
$\F(W,Q(T))$ is a $\K_{Q(T)}-Q(T)$-bimodule via
\begin{align}\label{E:Fbi}
  (a\cdot \psi\cdot q)(b) = \psi(q b a) = q \psi(ba)
\end{align}
for $\psi\in \F(W,Q(T))$, $q\in Q(T)$ and $a,b\in \K_{Q(T)}$.

Evaluation gives a perfect pairing $\ip{\cdot}{\cdot}: \K_{Q(T)} \times
\F(W,Q(T)) \longrightarrow Q(T)$ given by
\begin{equation*}
   \ip{a}{\psi} = \psi(a).
\end{equation*}
It is $Q(T)$-bilinear in the sense that
\begin{align*}
  \ip{q a}{\psi} = q \ip{a}{\psi} = \ip{a}{\psi\cdot q}
\end{align*}

Define the subring $\Psi\subset\F(W,Q(T))$ by
\begin{align}
\label{E:Psidef}
  \Psi &= \{ \psi\in \F(W,Q(T))\mid \psi(\K) \subset R(T)\} \\
\notag
  &= \{ \psi\in \F(W,Q(T))\mid \psi(T_w)\in R(T)\text{ for all $w\in W$}\}.
\end{align}
Clearly $\Psi$ is a $\K-R(T)$-bimodule.  By \eqref{E:Tbasis}, for
$v\in W$, there are unique elements $\psi^v\in \Psi$ such that
\begin{equation} \label{E:psidef}
  \psi^v(T_w) = \delta_{v,w}.
\end{equation}
for all $w\in W$.  We have $\Psi = \prod_{v\in W} R(T) \psi^v$.

\begin{rem} In Section \ref{S:KM} we show that
$\psi^v(w)$ is the restriction of the equivariant structure sheaf
$[\OO_{X_v}]$ of the finite-codimensional Schubert variety $X_v\subset
X$ of the thick Kac-Moody flag manifold $X$, to the $T$-fixed point
$w$.  See Section \ref{SS:KK} for the relationship between our
functions $\psi^v$ and the functions of \cite{KK:K}.
\end{rem}

Letting $w=\id$ we have
\begin{equation} \label{E:psiid}
  \delta_{v,\id} = \psi^v(T_\id) = \psi^v(\id).
\end{equation}

\begin{lem} \label{L:psiy}
For $v\in W$ and $i\in I$,
\begin{equation*}
  y_i \cdot \psi^v = \begin{cases}
  \psi^{vr_i} & \text{if $vr_i<v$}\\
  \psi^v &\text{if $vr_i>v$.}
  \end{cases}
\end{equation*}
\end{lem}
\begin{proof} For $w\in W$ we have
$T_w y_i = T_w (1+T_i)
  = \chi(wr_i>w)(T_w+T_{wr_i})$.
Therefore by~\eqref{E:Fbi}
\begin{align*}
  \ip{T_w}{y_i\cdot \psi^v} &=
  \ip{T_w y_i}{\psi^v} \\
  &= \chi(wr_i>w) \ip{T_w+T_{wr_i}}{\psi^v} \\
  &= \chi(wr_i>w) (\delta_{v,w}+\delta_{v,wr_i}) \\
  &= \chi(vr_i>v)\delta_{v,w} + \chi(v>vr_i) \delta_{vr_i,w}
\end{align*}
from which the Lemma follows.
\end{proof}

\begin{rem}\label{R:psirec} By \eqref{E:psiid} and Lemma \ref{L:psiy} we obtain the
following ``right hand" recurrence for $\psi^v(w)$:
\begin{enumerate}
\item If $w=\id$ then $\psi^v(\id) = \delta_{v,\id}$.
\item Otherwise let $i\in I$ be such that $wr_i <w$. Then
\begin{equation} \label{E:psirec}
  \psi^v(w) = \begin{cases}
  \psi^v(wr_i) &\text{if $v<vr_i$} \\
  (1-e^{-w(\al_i)}) \psi^{vr_i}(w) + e^{-w(\al_i)} \psi^v(wr_i)&\text{if $vr_i<v$.}
  \end{cases}
\end{equation}
This rewrites $\psi^v(w)$ in terms of $\psi^{v'}(w')$ for $(v',w')$ such that
either $w'<w$ or both $w'=w$ and $v'<v$.
\end{enumerate}
\end{rem}

\begin{lem} \label{L:psisupp} We have $\psi^v(w) = 0$ unless $v\le w$.
\end{lem}
\begin{proof} The statement is true for $w=\id$ by \eqref{E:psiid}.
Otherwise let $i\in I$ be such that $wr_i<w$ and suppose $v\not\le
w$. Then $v\not\le wr_i$ and $vr_i \not\le w$ (see \cite{Hum}). The
Lemma follows by induction using \eqref{E:psirec}.
\end{proof}

The next result follows from the definitions.
\begin{prop} \label{P:Kloc}
For all $v,w\in W$ we have $w = \sum_{v\in W} \psi^v(w) T_v$.
\end{prop}

\begin{rem}\label{R:psirecleft}
Proposition \ref{P:Kloc} leads to a ``left hand" recurrence for
$\psi^v(w)$.
\begin{enumerate}
\item For $w=\id$ we have \eqref{E:psiid}.
\item Otherwise let $i\in I$ be such that $r_iw<w$.
By induction on length we have
\begin{equation*}
\begin{split}
  w &= r_i(r_iw) \\
  &= (1 + (1-e^{\al_i})T_i) (\sum_u \psi^{u}(r_iw) T_u) \\
  &= \sum_u \psi^{u}(r_iw) T_u + (1-e^{\al_i}) \sum_u T_i \psi^{u}(r_iw) T_u \\
  &= \sum_u \psi^{u}(r_iw) T_u + (1-e^{\al_i}) \sum_u ((T_i\cdot \psi^{u}(r_iw))T_u
    + (r_i\cdot \psi^{u}(r_iw))T_iT_u ).
\end{split}
\end{equation*}
Taking the coefficient of $T_v$, we see that
\begin{equation*}
  \psi^{v}(w) = \psi^{v}(r_iw) + (1-e^{\al_i})\Bigl( (T_i\cdot \psi^{v}(r_iw)) +
  \chi(r_iv<v)\, r_i\cdot (\psi^{r_iv}(r_iw)-\psi^{v}(r_iw))
  \Bigr).
\end{equation*}
Therefore for $r_iw<w$ we have
\begin{equation*}
  \psi^{v}(w) =
  \begin{cases}
    r_i \cdot \psi^{v}(r_iw) & \text{if $r_iv>v$} \\
    e^{\al_i}\, r_i \cdot \psi^{v}(r_iw) + (1-e^{\al_i})\, r_i\cdot \psi^{r_iv}(r_iw) &\text{if $r_iv<v$.}
  \end{cases}
\end{equation*}
\end{enumerate}
\end{rem}

Define the inversion set of $v\in W$ by
\begin{equation*}
  \Inv(v) = \{ \al\in \Phi^{+\re} \mid r_\al v < v\}.
\end{equation*}

\begin{lem} \label{L:psidiag} For all $v\in W$, we have $\psi^v(v) = \prod_{\al\in\Inv(v)} (1-e^\al)$.
\end{lem}
\begin{proof} Follows directly from Remark \ref{R:psirecleft} and Lemma \ref{L:psisupp}.
\end{proof}

\begin{rem}\label{R:eta}
We have $\psi^v(w) = \eta(w \cdot \psi^{v^{-1}}(w^{-1}))$, where
$\eta:\Z[P] \to \Z[P]$ is given by $\eta(e^\la) = e^{-\la}$.
\end{rem}

\subsection{The GKM condition}
We recall the $K$-theoretic GKM (Goresky-Kottwitz-Macpherson)
condition as a criterion for membership in $\Psi$. This condition
and the associated geometry is discussed in
Subsection~\ref{SS:restriction}.

\begin{prop} \label{P:GKM} $\Psi$ is the set of $\psi\in \F(W,Q(T))$ such that
\begin{equation}\label{E:KGKM}
  \psi(r_\al w) - \psi(w) \in (1-e^\al) R(T)\qquad\text{for all $\al\in \Phi^{+\re}$ and $w\in W$.}
\end{equation}
\end{prop}

\begin{proof} Let $\beta=w^{-1}\al$ and $\psi\in\Psi$. Then $r_\al w=w r_\beta$ and
\begin{align*}
  \dfrac{\psi(r_\al w)-\psi(w)}{1-e^{\al}} &=
  \psi( (1-e^{\al})^{-1}(w r_\beta - w)) \\
  &= \psi(w T_\beta)
\end{align*}
which is in $R(T)$ since $wT_\beta\in \K$ using \eqref{E:Talconj}
and \eqref{E:r}.

For the converse, let $\psi\in \F(W,Q(T))$ satisfy \eqref{E:KGKM}
and suppose $\psi \neq 0$.
Let $v\in \Supp(\psi)$ be a minimal element. For every
$\al\in\Phi^{+\re}$ such that $r_\al v<v$ we have $\psi(v) \in
(1-e^\al)R(T)$ by \eqref{E:KGKM}, Lemma \ref{L:psisupp}, and the
minimality of $v$. Since the factors $(1-e^\al)$ are relatively
prime by~\cite[Proposition 6.3]{Kac}, $\psi(v) \in \psi^v(v)R(T)$ by
Lemma~\ref{L:psidiag}. Then $\psi'\in \Psi$ where
$\psi'(w)=\psi(w)-(\psi(v)/\psi^v(v)) \psi^v(w)$ for $w\in W$.
Moreover $v\not\in\Supp(\psi')$ and $\Supp(\psi')\setminus
\Supp(\psi)$ consists of elements strictly greater than $v$.
Repeating the argument for $\psi'$ and so on, we see that $\psi$ is
in $\prod_{v\in W} R(T)\psi^v$.
\end{proof}

\subsection{Structure constants and coproduct}
\label{SS:coprod}
The proof of the following result is straightforward but lengthy.

\begin{prop} \label{P:coprodaction} Let $M$ and $N$ be left $\K$-modules.
Define
\begin{equation*}
  M \otimes_{R(T)} N = (M\otimes_\Z N)/\langle qm\otimes n - m \otimes qn \mid
  \text{$q\in R(T)$, $m\in M$, $n\in N$} \rangle.
\end{equation*}
Then $\K$ acts on $M \otimes_{R(T)} N$ by
\begin{align*}
 q \cdot (m\otimes n) &= qm \otimes n \\
 T_i \cdot (m\otimes n) &= T_i \cdot m \otimes n + m \otimes T_i\cdot
  n + (1-e^{\al_i}) T_i\cdot m\otimes T_i\cdot n.
\end{align*}
Under this action we have
\begin{equation} \label{E:tensorWeyl}
  w \cdot (m\otimes n) = wm \otimes wn.
\end{equation}
\end{prop}

Consider the case $M=N=\K$. By Proposition \ref{P:coprodaction}
there is a left $R(T)$-module homomorphism $\Delta:\K \to \K
\otimes_{R(T)} \K$ defined by $\Delta(a) = a \cdot (1 \otimes 1)$. It satisfies
\begin{align}
  \Delta(q) &= q \otimes 1 &\qquad&\text{for $q\in R(T)$} \\
\label{E:DeltaT}
  \Delta(T_i) &= 1 \otimes T_i + T_i \otimes 1 + (1-e^{\al_i}) T_i \otimes T_i&\qquad&\text{for $i\in I$.}
\end{align}
Let $a\in\K$ and $\Delta(a)=\sum_{v,w} a_{v,w} T_v \otimes T_w$ with
$a_{v,w}\in R(T)$. It follows from Proposition \ref{P:coprodaction}
that the action of $a$ on $M \otimes_{R(T)} N$ can be computed in
the following simple ``componentwise" fashion: $a\cdot (m\otimes n)
= \sum_{v,w} a_{v,w} T_v m \otimes T_w n$.  In particular if
$b\in\K$ and $\Delta(b)=\sum_{v',w'} b_{v',w'} T_{v'} \otimes T_{w'}$ then
\begin{equation}\label{E:Deltamult}
\Delta(ab) =  \Delta(a)\cdot \Delta(b) := \sum_{v,w,v',w'} a_{v,w}
b_{v',w'} T_v T_{v'} \otimes T_w T_{w'}.
\end{equation}

\begin{rem} \label{R:Deltamult}
The naive componentwise product is ill-defined on all of $\K \otimes_{R(T)} \K$, for if
it were well-defined then $(T_i \otimes 1)(q\otimes 1)=(T_i\otimes1)(1\otimes q)$
or equivalently $T_iq\otimes 1 = T_i\otimes q = q(T_i\otimes 1)=qT_i \otimes 1$ which is false
for $q=e^{\al_i}$.
\end{rem}

There is a left $R(T)$-bilinear pairing $\ip{.}{.}:(\K \otimes_{R(T)}\K) \times (\Psi \otimes_{R(T)} \Psi)
\to R(T)$ given by
\begin{align*}
  \ip{a\otimes b}{\phi \otimes \psi} &= \ip{a}{\phi} \ip{b}{\psi}.
\end{align*}

\begin{lem}\label{L:pairprod}
For all $a\in\K$ and $\phi,\psi\in \Psi$, we have $\ip{a}{\phi\psi}
= \ip{\Delta(a)}{\phi\otimes \psi}$.
\end{lem}
\begin{proof}
First extend the definitions in the obvious manner to $\K_{Q(T)}$
and $\F(W,Q(T))$. Using left $Q(T)$-linearity we may then take
$a=w$. Then $\ip{\Delta(w)}{\phi\otimes\psi} = \ip{w\otimes
w}{\phi\otimes\psi} = \phi(w)\psi(w) = \ip{w}{\phi\psi}$.
\end{proof}

Define the structure ``constants" $c^{uv}_w\in R(T)$ by $\psi^u
\psi^v = \sum_{w\in W} c^{uv}_w \psi^w$.  The structure constants of
$\Psi$ are recovered by the map $\Delta$.
\begin{prop} \label{P:struct}
We have $\Delta(T_w)=\sum_{u,v} c_w^{uv} \, T_u \otimes T_v$ for all
$w \in W$.
\end{prop}
\begin{proof}
This follows from Lemma \ref{L:pairprod} and $\psi^u(T_v) = \delta_{uv}$.
\end{proof}

\subsection{Explicit localization formulae}
For the sake of completeness we give an explicit formula for the
values $\psi^v(w)$. It is a variant of a formula due independently
to Graham \cite{Gr} and Willems \cite{W}.  Let $\ve:\K_{Q(T)}\to
Q(T)$ be the left $Q(T)$-module homomorphism defined by $\ve(w)= 1$
for all $w\in W$.
\begin{prop} \label{P:psivalue} Let $v,w\in W$ and let $w=r_{i_1}r_{i_2}\dotsm r_{i_N}$
be any reduced decomposition of $w$. For $b_1b_2\dotsm b_N\in
\{0,1\}^N$, let $|b|=\sum_{i=1}^N b_i$. Then
\begin{equation} \label{E:psivalue}
  \psi^v(w) = \ve \sum_{b\in B(i_\bullet,v)} (-1)^{\ell(w)-|b|} \prod_{k=1}^N \begin{cases}
  (1-e^{\al_{i_k}}) r_{i_k} & \text{if $b_k=1$} \\
  r_{i_k} &\text{if $b_k=0$}
  \end{cases}
\end{equation}
where the sum runs over
\begin{equation}\label{E:sumset}
B(i_\bullet,v) = \left\{ b=(b_1,b_2,\dotsc,b_N)\in \{0,1\}^N \mid
\prod_{\substack{ k \\ b_k=1}} T_{i_k} = \pm T_v \right\}.
\end{equation}
\end{prop}

Formula~\eqref{E:psivalue} is the $K$-theoretic analogue of the
formula~\cite[Eq. (1.2)]{AJS,Billey} for the restriction of a
$T$-equivariant Schubert cohomology class $[X_v]$, to a $T$-fixed
point $w$.

\begin{ex} Let $G=SL_3$, $v=r_1$, and $w=r_1r_2r_1$.
Then there are three possible binary words $b$: $(1,0,0)$,
$(0,0,1)$, and $(1,0,1)$, yielding
\begin{align*}
  \psi^v(w) &= \ve \left((1-e^{\al_1})r_1 r_2 r_1 + r_1 r_2 (1-e^{\al_1}) r_1  -  (1-e^{\al_1})r_1 r_2 (1-e^{\al_1}) r_1 \right) \\
  &= (1-e^{\al_1})+(1-e^{r_1r_2 \al_1}) - (1-e^{\al_1})(1-e^{r_1r_2\al_1}) \\
  &= (1-e^{\al_1}) + (1-e^{\al_2}) - (1-e^{\al_1})(1-e^{\al_2}) \\
  &= 1 - e^{\al_1+\al_2}.
\end{align*}
Using instead the reduced decomposition $w=r_2r_1r_2$, there is only
one summand $b=(0,1,0)$ and we obtain
\begin{align*}
  \psi^v(w) = \ve (r_2 (1-e^{\al_1}) r_1 r_2) = 1-e^{r_2\al_1} = 1-e^{\al_1+\al_2}.
\end{align*}
\end{ex}

\begin{proof}[Proof of Proposition \ref{P:psivalue}] Let the right hand side of \eqref{E:psivalue} be denoted by $\phi^v(w)$.
We prove that $\psi^v(w)=\phi^v(w)$ by induction on $w$ and then on
$v$. Let $i=i_N$. If $v<vr_i$ then in any summand $b$, $b_N=0$, so
that $\phi^v(w)=\phi^v(wr_i)=\psi^v(wr_i)=\psi^v(w)$ by induction on the length
$\ell(w)$ of $w$ and \eqref{E:psirec}. Otherwise let $v>vr_i$. The part of
$\phi^v(w)$ with $b_N=0$, is given by $\phi^v(wr_i)=\psi^v(wr_i)$.
The rest of $\phi^v(w)$ consists of summands with $b_N=1$. Consider
the left-to-right product $\pm T_u$ of $T_{i_k}$ for which $b_k=1$
except that the term $b_N=1$ is omitted. For the $b$ such that
$u=v$, the last factor $T_{i_N}=T_i$ produces an additional negative
sign and we obtain $-(1-e^{wr_i(\al_i)})
\phi^v(wr_i)=(e^{-w(\al_i)}-1)\psi^v(wr_i)$ because the product of
the reflections $r_{i_1}\dotsm r_{i_{N-1}}$ is $wr_i$. For the $b$
with $u=vr_i$, we obtain $(wr_i\cdot (1-e^{\al_i})) \phi^{vr_i}(w)=
(1-e^{-w(\al_i)}) \psi^{vr_i}(w)$. In total we obtain the right hand
side of \eqref{E:psirec}, which equals $\psi^v(w)$.
\end{proof}

\section{Equivariant $K$-cohomology of Kac-Moody flag manifolds}
\label{S:KM}

In contrast to \cite{KK:K} which uses the ``thin" Kac-Moody flag manifold $X_{\rm ind}$
which is an ind-scheme with finite-dimensional Schubert varieties \cite{Ku},
we employ the larger ``thick" Kac-Moody flag manifold $X$ \cite{Kash},
which is a scheme of infinite type with finite codimensional Schubert varieties.
Using the thick Kac-Moody flag manifold,
we give natural geometric interpretations to the constructions in the $K$-NilHecke ring.

\subsection{Kac-Moody thick flags}\label{SS:flags}
For the following discussion see \cite{Kash}. Let $T$ be the
algebraic torus with character group $P$, $U_\pm$ the group scheme
with $\Lie(U_\pm)=\nn_\pm$, $B_\pm $ the Borel subgroups with
$\Lie(B_\pm)=\ttt\oplus \nn_\pm$, and for $i\in I$ let $P_i^\pm$ be
the parabolic group with $\Lie(P_i^\pm) = \Lie(B_\pm) \oplus
\geh_{\pm\al_i}$. These groups are all contained in an affine scheme
$G_\infty$ of infinite type, that contains a canonical ``identity"
point $e$. Let $G\subset G_\infty$ be the open subset defined by $G=
\bigcup P_{i_1} \dotsm P_{i_m} e P_{j_1}^- P_{j_2}^-\dotsm
P_{j_m}^-$. It is not a group but admits a free left action by each
$P_i$ and a free right action by each $P_j^-$.

Then $X=G/B_-$ is the thick Kac-Moody flag manifold; it
is a scheme of infinite type over $\C$.
For each subset $J\subset I$ let $P_J^-\subset G$
be the group generated by $B_-$ and $P_i^-$ for $i\in J$.

Write $X^J = G/P_J^-$. Let $W_J\subset W$ be the subgroup
generated by $r_i$ for $i\in J$ and let $W^J$ be the set of
minimal length coset representatives in $W/W_J$. For $w\in W^J$
let $\mathring{X}^J_w=B w P_J^-/P_J^-$ where $B=B_+$; it is locally closed in $X^J$.
We have the $B$-orbit decomposition
\begin{equation*}
  X^J = \bigsqcup_{w\in W^J} \mathring{X}^{J}_w.
\end{equation*}
Let $X_w^J = \overline{\mathring{X}^{J}_w}$ be the Schubert variety.
It has codimension $\ell(w)$ in $X^J$ and coherent structure sheaf
$\OO_{X_w^J}$. We have
\begin{equation} \label{E:Schubvarcell}
   X_w^J = \bigsqcup_{\substack{v\in W^J\\ v\ge w}} \mathring{X}^{J}_v.
\end{equation}

Let $S$ be a finite Bruhat order ideal of $W^J$ (a finite subset
$S\subset W^J$ such that if $v\in S$, $u\in W^J$, and $u\le v$, then
$u\in S$). Let $\Omega_S^J = \bigsqcup_{w\in S} \mathring{X}_w^{J} =
\bigcup_{w\in S} w B P_J^-/P_J^-$ be a $B$-stable finite union of
translations of the big cell $\mathring{X}^{J}:=BP_J^-/P_J^-$, which
is open in $X^J$. The big cell is an affine space of countable
dimension (finite if $\geh$ is finite-dimensional):
$\mathring{X}^{J}\cong\mathrm{Spec}(\C[x_1,x_2,\dotsc])$.

\subsection{Equivariant $K$-cohomology}
Denote by $\Coh^T(\Omega_S^J)$ the category of coherent
$T$-equivariant $\OO_{\Omega_S^J}$-modules and let $K^T(\Omega_S^J)$
be the Grothendieck group of $\Coh^T(\Omega_S^J)$. For each $w\in
S$, $\OO_{X_w^J}\in \Coh^T(\Omega_S^J)$ and therefore defines a
class $[\OO_{X_w^J}]\in K^T(\Omega_S^J)$. Define
\begin{equation*}
  K^T(X^J) = \varprojlim_S K^T(\Omega_S^J).
\end{equation*}
One may show (as in \cite{KS} for the case $J=\emptyset$) that
\begin{equation} \label{E:KTbasis}
  K^T(X^J) \cong \prod_{w\in W^J} K^T(\pnt) [\OO_{X_w^J}].
\end{equation}
Recall that $K^T(\pnt)\cong R(T)\cong\Z[P]=\bigoplus_{\la\in P} \Z
e^\la$.  Elements of $K^T(X^J)$ are possibly infinite
$K^T(\pnt)$-linear combinations of equivariant Schubert classes
$[\OO_{X_w^J}]$.

\subsection{Restriction to $T$-fixed points}
\label{SS:restriction}

For $w\in W^J$ let $i_w^J:\{\pnt\}\to W^J \cong (X^J)^T \subset X^J$ be the inclusion with image $\{w P_J^-/P_J^-\}$.
Restriction to the $T$-fixed points
induces an injective $R(T)$-algebra homomorphism \cite{HHH,KK:K}
\begin{equation} \label{E:Kloc}
\begin{split}
  K^T(X^J) &\overset{\res^J}\longrightarrow K^T((X^J)^T) \cong K^T(W^J) \cong \F(W^J,R(T)) \\
  c &\longmapsto (w\longmapsto i_w^{J*}(c))
\end{split}
\end{equation}
where $\F(W^J,R(T))$ is the $R(T)$-algebra of functions $W^J\to R(T)$
with pointwise multiplication and $R(T)$-action
\begin{equation*}
(q\psi)(w)=q \psi(w)
\end{equation*}
for $q\in R(T)$, $\psi\in \F(W^J,R(T))$, and $w\in W^J$.  Let
$\iota_J:\F(W^J,R(T))\to\F(W,R(T))$ be defined by extending
functions to be constant on cosets in $W/W_J$:
\begin{equation*}
  \iota_J(\psi)(w) = \psi(w')
\end{equation*}
for $\psi\in \F(W^J,R(T))$, where $w'\in W^J$ is such that
$w'W_J=wW_J$.

Define $\Psi^J\subset \F(W,R(T))$ by $\psi\in \Psi^J$ if and only if
$\psi$ is in the image of $\iota_J$, and
\begin{equation} \label{E:GKM}
  \psi(r_\al w)-\psi(w)\in (1-e^\al) R(T)\qquad\text{for all $w,r_\al w\in W$, $\al\in \Phi^\re$.}
\end{equation}
We call this the GKM condition\footnote{The corresponding criterion
was proved in \cite{GKM} for equivariant cohomology for more general
spaces, commonly called GKM spaces.  For Kac-Moody flag ind-schemes
the criterion follows directly from results in \cite{KK:K}.  See
also \cite{HHH} for more general cohomology theories and spaces. }
for $K^T(X^J)$.

\begin{thm}[\cite{KK:K,HHH}] \label{T:GKM} $K^T(X^J) \cong \Psi^J$.
\end{thm}

For the sake of completeness we include a proof of Theorem
\ref{T:GKM}.  For $v\in W^J$ define $\psi^v_J\in \F(W^J,R(T))$ to be
the image of $[\OO_{X_v^J}]$:
\begin{equation} \label{E:psilocdef}
  \psi^v_J(w) = i_w^{J*}([\OO_{X_v^J}])\qquad\text{for $v,w\in W^J$.}
\end{equation}

When $J=\emptyset$ we shall write $X=X^{\emptyset}$,
$\Psi=\Psi^{\emptyset}$, and so on, suppressing $\emptyset$ in the
notation. Observe that the definition of $\Psi=\Psi^\emptyset$ in
this section agrees with the definition \eqref{E:Psidef} by
Proposition \ref{P:GKM}. Provisionally for $J=\emptyset$ we write
$\psi^v_\emptyset$ for the functions defined by \eqref{E:psilocdef},
and show they agree with the functions defined by \eqref{E:psidef}
using the $K$-NilHecke ring.

\begin{thm} \label{T:psiloc} For all $v\in W$, we have $\psi^v = \psi^v_\emptyset$.
\end{thm}

\subsection{Push-pull and $y_i$}

Fix $i\in I$. For the singleton $J=\{i\}$ let $P_i^-=P_J^-$, $X^i=X^J$,
and let $p_i:X\to X^i$ be the projection, which is a $\Po$-bundle.
In \cite{KS} it is shown that
\begin{equation} \label{E:Opushpull}
  p_i^*p_{i\,*}([\OO_{X_v}]) =
  \begin{cases}
  [\OO_{X_{vr_i}}] & \text{if $vr_i<v$} \\
  [\OO_{X_v}] &\text{if $vr_i>v$.}
  \end{cases}
\end{equation}

\begin{prop} \label{P:pushpull}
The map $\psi\mapsto y_i \cdot \psi$, is an $R(T)$-module
endomorphism of  $\F(W,R(T))$ such that the diagram commutes:
\begin{equation*}
\begin{diagram}
\node{K^T(X)}  \arrow{e,t}{\res} \arrow{s,t}{p_i^*p_{i\,*}} \node{\F(W,R(T))} \arrow{s,b}{y_i\cdot-} \\
\node{K^T(X)} \arrow{e,b}{\res} \node{\F(W,R(T)).}
\end{diagram}
\end{equation*}
\end{prop}
\begin{proof}
Let $x_i^o\in X_i$ be the point $P_i^-/ P_i^-$.
Let $G_i\supset G_i^+\supset T$
be the subgroups with $\Lie(G_i)=\ttt\oplus\geh_{\al_i}\oplus
\geh_{-\al_i}$ and $\Lie(G_i^+)=\ttt\oplus\geh_{\al_i}$.

Now let $w\in W$ with $wr_i<w$.
Then $\Ad(w)G_i^+\subset B$
and it stabilizes $wx_i^o$ and $p_i^{-1}(wx_i^o)$.
Let $j:p_i^{-1}(wx_i^o)\to X$ be the inclusion.
Then for $\mathcal{F}\in \Coh^B(X)$, the first left derived functor
$L_1 j^*\mathcal{F}$ is $\Ad(w)G_i^+$-equivariant.
For $x\in \{w,wr_i\}$ let $i'_x$ be the inclusion
$\{\pnt\} \to x x_i^o\subset p_i^{-1}(wx_i^o)$.
Hence we have the commutative diagram
\begin{equation*}
\begin{diagram}
\node{K^B(X)} \arrow[2]{e,t}{i_x^*} \arrow{se,b}{j^*}\node[2]{K^T(\pnt)} \\
\node[2]{K^{\Ad(w)G_i^+}(p_i^{-1}(wx_i^o))} \arrow{ne,b}{i_x^{\prime*}}
\end{diagram}
\end{equation*}
and the isomorphisms that forget down to the Levi
\begin{align*}
K^B(X) &\cong K^T(X) \\
K^{\Ad(w)G_i^+}(p_i^{-1}(wx_i^o))&\cong K^T(p_i^{-1}(wx_i^o))
\end{align*}
which allows the reduction to the case of $p_i^{-1}(wx_i^o)\cong \Po$, where the result is standard;
see \cite[Cor. 6.1.17]{CG}.
\end{proof}

\begin{proof}[Proof of Theorem~\ref{T:psiloc}]
We show that the functions $\psi_\emptyset^v$ satisfy the recurrence in Remark \ref{R:psirec}.

Since $X_\id = X$, it follows that $\psi_\emptyset^\id(\id)=1$.

Let $v\in W$. By \eqref{E:Schubvarcell}, we have $w\in X_v$ if and only if $w\ge v$. Therefore
\begin{equation*}
\psi^v_\emptyset(w) = 0 \qquad\text{unless $v\le w$.}
\end{equation*}
In particular $\psi_\emptyset^v(\id)=0$ if $v\ne \id$. Therefore the values $\psi_\emptyset^v(w)$
satisfy the base case of the recurrence. To show that \eqref{E:psirec} holds for $\psi_\emptyset^v$,
it is equivalent to show
\begin{equation*}
  y_i \cdot \psi_\emptyset^v = \begin{cases}
  \psi_\emptyset^{vr_i} & \text{if $vr_i<v$,} \\
  \psi_\emptyset^v &\text{if $vr_i>v$.}
  \end{cases}
\end{equation*}
But this holds by Proposition \ref{P:pushpull} and \eqref{E:Opushpull}.
\end{proof}

Let $p_J:X\to X^J$ be the projection.  Then there is a commutative
diagram of injective $R(T)$-algebra maps, where the horizontal maps
are restriction maps as in \eqref{E:Kloc}
\begin{equation*}
\begin{diagram}
\node{K^T(X^J)}  \arrow{e,t}{\res^J} \arrow{s,t}{p_J^*} \node{\F(W^J,R(T))} \arrow{s,b}{\iota_J} \\
\node{K^T(X)} \arrow{e,b}{\res} \node{\F(W,R(T)).}
\end{diagram}
\end{equation*}
We have $p_J^*([\OO_{X_v^J}]) = [\OO_{X_v}]$ and
$\iota_J(\psi^v_J)=\psi^v$ for $v \in W^J$.

\begin{proof}[Proof of Theorem \ref{T:GKM}]
For the case $J = \emptyset$, the result follows from
\eqref{E:KTbasis} and Theorem~\ref{T:psiloc} and Proposition~\ref{P:GKM}.
For $J \neq \emptyset$, let $\psi = \sum_w a_w \psi^w$ be in the
image of $\iota_J$.  It suffices to show that $a_w = 0$ for $w
\notin W^J$. Since $T_i$ is a $Q(T)$-multiple of $r_i - 1$, we have
$T_i \cdot \psi = 0$ whenever $i \in J$.  Suppose $a_w \neq 0$ for
some $w \notin W^J$. Pick such a $w$ with minimal length, and let $i
\in J$ be such that $w r_i < w$.  By Lemma \ref{L:psiy}, and $T_i =
y_i - 1$, we deduce that the coefficient of $\psi^{wr_i}$ in $T_i
\cdot \psi$ is non-zero, a contradiction.
\end{proof}

\section{Affine flag manifold and affine Grassmannian}
\label{S:P} We now specialize our constructions to the case of an
affine root system, and consider the thick affine flag manifold
$X_\af$ and thick affine Grassmannian $\Gr_G$ and their equivariant
$K$-cohomology.  However, instead of using the full affine torus
$T_\af\subset G_\af$, we use the torus $T\subset G$ and consider
$K^T(\Gr_G)$.  We give a {\it small torus GKM condition}, which is the
$K$-theoretic analogue of a result of Goresky, Kottwitz, and MacPherson
\cite{GKM:2004} in cohomology.

\subsection{Affine flag manifold}
We fix notation special to affine root systems and the associated finite root system.

Let $\geh\supset\bb\supset \ttt$ be a simple Lie algebra over $\C$, Borel subalgebra,
and Cartan subalgebra, Dynkin node set $I$, finite Weyl group $W$, simple reflections $\{r_i\mid i\in I\}$,
weight lattice $P=\bigoplus_{i\in I} \Z \omega_i\subset \ttt^*$ with fundamental weights $\omega_i$,
root lattice $Q=\bigoplus_{i\in I} \Z \alpha_i\subset \ttt^*$ with simple roots $\al_i$,
coroot lattice $Q^\vee = \bigoplus_{i\in I} \Z \al_i^\vee\subset \ttt$.

Let $G\supset B\supset T$ be a simple and simply-connected algebraic group over $\C$
with $\Lie(G)=\geh$, with Borel subgroup $B$ and maximal algebraic torus $T$.

Let $\geh_\af = (\C[t,t^{-1}] \otimes \geh) \oplus \C c \oplus \C d$
be the untwisted affine Kac-Moody algebra with canonical simple
subalgebra $\geh$, canonical central element $c$ and degree
derivation $d$. Let $\geh_\af = \nn_\af^+ \oplus \ttt_\af \oplus
\nn_\af^-$ be the triangular decomposition with affine Cartan
subalgebra $\ttt_\af$. Let $I_\af=\{0\}\cup I$ be the affine Dynkin
node set. Let $\{a_i\in\Z_{>0}\mid i\in I_\af\}$ be the unique
collection of relatively prime positive integers giving a dependency
for the columns of the affine Cartan matrix $(a_{ij})_{i,j\in
I_\af}$. Then $\delta=\sum_{i\in I_\af} a_i \al_i$ is the null root.
The affine weight lattice is given by $P_\af = \Z \delta \oplus
\bigoplus_{i\in I_\af} \Z \La_i\subset\ttt_\af^*$ where $\{\La_i\mid
i\in I_\af\}$ are the affine fundamental weights. Let $Q_\af$ and
$Q_\af^\vee$ be the affine root and coroot lattices. Let $W_\af$ be
the affine Weyl group, with simple reflections $r_i$ for $i\in
I_\af$. Considering the subset $I$ of $I_\af$, $W=W_I\subset W_\af$
and $W_\af^I$ is the set of minimal length coset representatives in
$W_\af/W$. We have $W_\af \cong W \ltimes Q^\vee$ with $\la\in
Q^\vee$ written $t_\la\in W_\af$. There is a bijection $\Wz\to
Q^\vee$ sending $w\in\Wz$ to $\la\in Q^\vee$, where $\la$ is defined
by $wW = t_\la W$. Let $\la^-\in W\cdot \la$ be antidominant and
$u\in W$ shortest such that $u \la^- = \la$. Then $w=t_\la u$ and
$\ell(w)=\ell(t_\la)-\ell(u)$.

Let $\Phi_\af=\Z\delta\cup(\Z\delta+\Phi)$ be the set of affine roots.
The affine real roots are given by $\Phi_\af^\re = \Z\delta+\Phi$.
Let $\Phi_\af^{\pm}$ be the sets of positive and negative affine roots.
The set of positive affine real roots is defined by $\Phi_\af^{+\re} = \Phi_\af^+\cap \Phi_\af^\re=
\Phi^+ \cup (\Z_{>0}\delta + \Phi)$. A typical real root $\al+m\delta\in \Phi_\af^\re$
for $\al\in \Phi$ and $m\in\Z$, has associated reflection
$r_{\al+m\delta} = r_\al t_{m\al^\vee}\in W_\af$.

Let $G_\af\supset P_I^-\supset B_\af^-\supset T_\af$ be the schemes
of Subsection \ref{SS:flags} associated with $\geh_\af$, $P_I^-$ the
maximal parabolic group scheme for the subset of Dynkin nodes
$I\subset I_\af$, $B_\af^-$ the negative affine Borel group,
$X_\af=G_\af/B_\af^-$ the thick affine flag manifold,
$\Gr_G=X_\af^I=G_\af/P_I^-$ the thick affine Grassmannian.

\subsection{Equivariant $K$-theory for affine flags with small torus action}
Following Peterson \cite{P} and Goresky, Kottwitz, and MacPherson
\cite{GKM:2004} in the cohomology case, we consider the action of
the smaller torus $T=T_\af\cap G$. The goal is to formulate and
prove the analogue of Theorem \ref{T:GKM} for $K^T(X_\af)$ and
$K^T(\Gr_G)$.  We let $\Psi'_\af \subset \F(W_\af,R(T_\af))$ denote
the ring defined by \eqref{E:Psidef} for the affine Lie algebra
$\geh_\af$, so that $\Psi'_\af \cong K^{T_\af}(X_\af)$.

The natural projection of weight lattices $P_\af\to P$ is surjective
with kernel $\Z\delta \oplus \Z \Lambda_0$. It induces the
projections
\begin{equation*}
\begin{diagram}
\node{\Z[P_\af]} \arrow{e,t}{\phi} \node{\Z[P]} \arrow{e,t}{\phi_0} \node{\Z}
\end{diagram}
\end{equation*}
and a commutative diagram
\begin{equation}\label{E:forgetdiag}
\begin{diagram}
\node{K^{T_\af}(X_\af)} \arrow{e,t}{\res'} \arrow{s,t}{\For} \node{\Psi'_\af } \arrow{s,b}{\phi\circ -} \\
\node{K^T(X_\af)} \arrow{e,b}{\res} \node{\F(W_\af,R(T))}
\end{diagram}
\end{equation}
where the horizontal maps are restriction to
$W_\af=X_\af^{T_\af}\subset X_\af^T$ and the map $\For$ regards a
$T_\af$-equivariant $O_{X_\af}$-module as a $T$-equivariant one.  We
change notations slightly and denote the Schubert classes by
$\psi'^v \in \Psi'_\af$ and define $\psi^v := \phi \circ \psi'^v \in
\F(W_\af,R(T))$.

The following definition is inspired by the analogous cohomological condition in
\cite[Theorem 9.2]{GKM:2004}.  A function $\psi \in \F(W_\af,R(T))$, can be extended
by linearity to give a function $\psi' \in \F(\oplus_{w \in W_\af}R(T)\cdot w, R(T))$.  In the following definition
we abuse notation by identifying $\psi$ with $\psi'$.
\begin{definition} \label{D:smallGKM}
We say that $\psi \in  \F(W_\af,R(T))$ satisfies the {\it small
torus Grassmannian GKM condition} if
\begin{equation} \label{E:smallGrassGKM}
\psi((1-t_{\alpha^\vee})^d w) \in (1-e^\alpha)^d R(T)\quad\text{for all $d\in\Z_{>0}$, $w\in W_\af$, and $\alpha\in \Phi$.}
\end{equation}
We say that $\psi \in  \F(W_\af,R(T))$ satisfies the {\it small
torus GKM condition} if, in addition to~\eqref{E:smallGrassGKM}, we
have
\begin{equation} \label{E:smallGKM}
\psi((1-t_{\alpha^\vee})^{d-1} (1 - r_\alpha) w) \in (1-e^\alpha)^d R(T)\quad\text{for all $d\in\Z_{>0}$, $w\in W_\af$, and $\alpha\in\Phi$.}
\end{equation}
\end{definition}
Let $\Psi_\af$ be the set of $\psi\in \F(W_\af,R(T))$ that satisfy
the small torus GKM condition and $\Psi_\af^{I}$ the set of
$\psi\in\F(W_\af,R(T))$ that are constant on cosets $w W$ for $w\in
W_\af$ and satisfy the small torus Grassmannian GKM condition.

\begin{lem} \label{L:Jcongruent}
Suppose $\psi$ satisfies~\eqref{E:smallGrassGKM} and let $J:=(1-e^\al)^d R(T)$. Then
\begin{equation*}
    \psi((1-t_{\al^\vee})^{d-1}w) \quad \text{and} \quad \psi((1-t_{\al^\vee})^{d-1} t_{p\al^\vee} w)
\end{equation*}
are congruent modulo the ideal $J$ for all $p\in \Z$.
\end{lem}
\begin{proof}
By \eqref{E:smallGrassGKM}
\begin{equation*}
  J \ni \psi((1-t_{\al^\vee})^d w) = \psi((1-t_{\al^\vee})^{d-1}w)- \psi((1-t_{\al^\vee})^{d-1} t_{\al^\vee} w),
\end{equation*}
so that the Lemma holds for $p=1$. Repeating the same argument for
$\psi((1-t_{\al^\vee})^{d-1} t_{\al^\vee} w)$ yields the Lemma for all $p\in\Z_{\ge0}$.
Replacing $w$ by $t_{-p\al^\vee}w$ yields the statement for all $p\in \Z$.
\end{proof}

\begin{thm} \label{T:GKMPet} \mbox{}
\begin{enumerate}
\item $K^T(X_\af) \cong \Psi_\af = \prod_{v\in W_\af} R(T)\,\psi^v$.
\item $K^T(\Gr_G) \cong \Psi_\af^{ I} = \prod_{v\in \Wz} R(T)\,  \psi^v$.
\end{enumerate}
\end{thm}
\begin{proof}
Due to Lemmata \ref{L:psisupp} and \ref{L:psidiag}, the set
$\{\psi^v \mid v \in W_\af\}$ is independent over $R(T)$.
Arguing as in \cite{KS} one may show that $K^T(X_\af)$ consists of
possibly infinite $R(T)$-linear combinations of the $[\OO_{X_v}]$.
By the commutativity of the diagram \eqref{E:forgetdiag}, we
conclude that
\begin{equation*}
  K^T(X_\af) \cong \prod_{v\in W_\af} R(T) [\OO_{X_v}],
\end{equation*}
the map $\For$ is surjective, and that $\res'$ is injective with
image $\prod_{v\in W_\af} R(T) \psi^v$. For (1) it remains to show
that $\Psi_\af = \prod_{v \in W_\af} R(T) \psi^v$.
Let $v\in W_\af$. We first show that $\psi^v\in \Psi_\af$. Let $w\in
W_\af$, $\al\in\Phi$, and $d\in\Z_{>0}$. Let $W'\subset W_\af$ be
the subgroup generated by $t_{\al^\vee}$ and $r_\al$; it is
isomorphic to the affine Weyl group of $SL_2$. Define the function
$f:W'\to R(T)$ by $f(x)=\psi'^v(x w)$. Since $\psi'^v$ satisfies the
big torus GKM condition \eqref{E:GKM} for $X_\af$, $f$ satisfies
\eqref{E:GKM} for a copy of the $SL_2$ affine flag variety $X'$.
Therefore $f$ is a possibly infinite $R(T)$-linear combination of
Schubert classes in $X'$. By Propositions \ref{P:sl2 Grassmannian
small GKM} and \ref{P:sl2smallGKM}, proved below, $\phi \circ f$
satisfies the small torus GKM condition for $X'$. It follows that
$\psi^v\in \Psi_\af$.

Conversely, suppose $\psi\in \Psi_\af$. We show that $\psi\in
\prod_{v\in W_\af} R(T) \psi^v$.
Let $x=t_\la u\in \Supp(\psi)$ be of minimal length, with $u\in W$
and $\la\in Q^\vee$. It suffices to show that
\begin{equation*}
    \psi(x) \in \psi^x(x) R(T),
\end{equation*}
for defining $\psi'\in\Psi_\af$ by
$\psi'=\psi-\dfrac{\psi(x)}{\psi^x(x)}\psi^x$,
we have $\Supp(\psi')\subsetneq \Supp(\psi)$, and repeating, we may
write $\psi$ as a $R(T)$-linear combination of the $\psi^x$.

The elements $\{1-e^\al\mid \al\in \Phi^+\}$ are relatively prime in
$R(T)$. Letting $\al\in\Phi^+$, by Lemma \ref{L:psidiag} it suffices
to show that
\begin{equation*}
  \psi(x)\in J := (1-e^\al)^d R(T)
\end{equation*}
where $d=|\Inv_\al(x)|$ where $\Inv_\al(x)$ is the set of roots in $\Inv(x)$ of the form $\pm \al + k \delta$
for some $k\in\Z_{\ge0}$.
Note that for $\beta\in\Phi_\af^{+\re}$, $\beta\in\Inv(x)$ if and only if
$x^{-1}\cdot \beta\in - \Phi_\af^{+\re}$. We have
\begin{align*}
  x^{-1} \cdot (\pm \al + k \delta) &= u^{-1} t_{-\la} \cdot (\pm \al + k\delta) = \pm u^{-1} \al
  + (k \pm \ip{\la}{\al}) \delta.
\end{align*}
Hence we have
\begin{align*}
  \Inv_\al(x) = \begin{cases}
\{\al,\al+\delta,\dotsc,\al-(\ip{\la}{\al} +\chi(\al\not\in \Inv(u)))\delta \} &\text{if $\ip{\la}{\al} \le 0$} \\
\{-\al+\delta,-\al+2\delta,\dotsc,-\al+(\ip{\la}{\al}-\chi(\al\in\Inv(u)))\delta \} &\text{if $\ip{\la}{\al} > 0$.}
\end{cases}
\end{align*}
Suppose first that $\ip{\la}{\al}>0$.  Then
$d=\ip{\la}{\al}-\chi(\al\in\Inv(u))$.  Applying~\eqref{E:smallGKM} to $y=t_{(1-d)\al^\vee}x$,
we have $Z_1\in J$ where
\begin{align*}
  Z_1&=\psi((1-t_{\al^\vee})^{d-1}(1-r_\al) y) \\
  &=(-1)^{d-1} \psi((1-t_{-\al^\vee})^{d-1}x) - \psi((1-t_{\al^\vee})^{d-1} r_\al y) \\
  &= (-1)^{d-1} \psi(x) - \psi((1-t_{\al^\vee})^{d-1} r_\al y)
\end{align*}
where the last equality holds by the assumption on $\Supp(\psi)$ and
a calculation of $\Inv_\al(r_{\alpha+d\delta} x)$, giving $x >
r_{\alpha+d\delta} x > t_{-k\al^\vee}x$ for all $k \in [1,d-1]$.
By Lemma~\ref{L:Jcongruent} we have $Z_2\in J$ where
\begin{align*}
Z_2 = \psi((1-t_{\al^\vee})^{d-1} r_\al y) - \psi((1-t_{\al^\vee})^{d-1} t_{p\al^\vee} r_\al y)
\end{align*}
for any $p\in\Z$. Thus
\begin{equation*}
 Z_1+Z_2 = (-1)^{d-1} \psi(x) -  \psi((1-t_{\al^\vee})^{d-1} t_{p\al^\vee}r_\al y) \in J.
\end{equation*}
By the assumption on $\Supp(x)$ and the calculation of
$\Inv_\al(x)$,
\begin{align*}
  \psi((1-t_{\al^\vee})^{d-1} t_{p\al^\vee} r_\al y) = 0
\end{align*}
for $p=2-d$. It follows that $\psi(x)\in J$.

Otherwise $\ip{\la}{\al}\le 0$. By the previous case we may assume
that $t_{d\al^\vee}x\not\in\Supp(\psi)$.  Thus
\begin{align*}
  \psi(x) = \psi((1-t_{\al^\vee})^d x) \in J
\end{align*}
by induction on $\Supp(\psi)$ and \eqref{E:smallGrassGKM}.
This proves (1).

For (2) it suffices to show that $\psi\in \Psi_\af^{I}$ if and only
if $\psi\in\prod_{v\in W_\af^I} R(T) \psi^v$. First let $\psi\in
\Psi_\af^{I}$. Let $w=t_\la u\in W_\af$ with $\la\in Q^\vee$ and
$u\in W$, $\al\in\Phi$. We verify that $\psi$ satisfies the small
torus GKM condition. We have
\begin{align*}
  r_\al w = r_\al t_\la u = t_{r_\al(\la)} r_\al u = t_{-\ip{\la}{\al}\al^\vee} t_\la r_\al u.
\end{align*}
Since by assumption $\psi$ is constant on cosets $W_\af/W$ we have
\begin{align*}
  \psi((1-t_{\al^\vee})^{d-1} (1-r_\al) w) &=
  \psi((1-t_{\al^\vee})^{d-1} (1 - t_{-\ip{\la}{\al}\al^\vee}) t_\la).
\end{align*}
But $1-t_{k \al^\vee}$ is divisible by $1-t_{\al^\vee}$ for any
$k\in \Z$. Therefore $\psi$ satisfies the small torus GKM condition
because it satisfies the Grassmannian one. Part (1) and the fact
that $\psi$ is constant on cosets $W_\af/W$ implies that
$\psi\in\prod_{v\in W_\af^I} R(T)\psi^v$.

Conversely, it suffices show that for every $v\in W_\af^I$,
$\psi^v\in \Psi_\af^{I}$. But this follows by part (1) and the fact
that for such $v$, $\psi^v$ is constant on cosets $W_\af/W$.
\end{proof}

\subsection{Small torus GKM condition for $\slh_2$} \label{SS:small GKM sl2}
In this section we prove that the Schubert classes $\psi^v$ for
$\slh_2$ satisfy the small torus GKM condition of
Definition~\ref{D:smallGKM}. To this end, we first derive explicit
expressions for $\psi^v$.

For $i\in \Z_{\ge 0}$ let
\begin{equation} \label{E:sl2gens}
\begin{aligned}
\sigma_{2i} &= (r_1r_0)^i &\qquad \sigma_{-2i}&=(r_0r_1)^i,\\
\sigma_{2i+1}&=r_0\sigma_{2i} & \sigma_{-(2i+1)}&=r_1 \sigma_{-2i}.
\end{aligned}
\end{equation}
Then we have $\ell(\sigma_j) = |j|$ for $j\in\Z$, $W_\af^I=\{\sigma_j\mid j\in\Z_{\ge0}\}$, and
\begin{equation*}
  \sigma_{2i} = t_{-i\al^\vee}\qquad\text{for $i\in\Z$.}
\end{equation*}

Let $\psi^i_j:=\psi^{\sigma_i}(\sigma_j)$ for $i,j\in\Z$, where we have set $\delta=0$.
We set $x=e^\alpha$ and let $S_{\le a}$ be the sum of homogeneous symmetric functions
$h_0+h_1+\dotsm+h_a$. Let $S_{\le a}^i(x)$ mean $S_{\le a}[x,x,\dotsc,x]$ where
there are $i$ copies of $x$. For $i,a \in \Z$ such that $a,m\ge 0$ we have
\begin{align}
\label{E:psi explicit1}
  \psi^m_{2i+2a} &= (1-x)^m S_{\le a}^m(x) = \psi^m_{-2i-2a-1}  && \text{for $m=2i$ or $2i-1$,}\\
  \label{E:psi explicit2}
  \psi^m_{2i+2a+1} &= (1-x^{-1})^m S_{\le a}^m(x^{-1}) =  \psi^m_{-2i-2a-2}
   && \text{for $m=2i$ or $2i+1$,}
\end{align}
and zero otherwise. Furthermore
\begin{equation*}
    \psi_{-i}^{-m}(x) = \psi_i^m(x^{-1}).
\end{equation*}
These are easily proved by induction using the left- and right-hand recurrence for the
localization of Schubert classes, together with the recurrence
\begin{equation*}
  S_{\le a}^i(x) = x S_{\le a-1}^i(x) + S_{\le a}^{i-1}(x).
\end{equation*}
We also have the explicit formula
\begin{equation} \label{E:Sle}
  S_{\le a}^i(x) = \sum_{j=0}^a x^j \binom{j+i-1}{i-1}.
\end{equation}

\begin{prop} \label{P:sl2 Grassmannian small GKM}
For all $d\ge 1$, $m\in \Z$, and $w\in W_\af$ we have
\begin{equation*}
\psi^m((1-t_{\alpha^\vee})^d w) \in (1-x)^d \Z[x^\pm].
\end{equation*}
\end{prop}

\begin{proof}
We prove the claim for $m=2i$ and for the ranges from $t_{(-i-a)\alpha^\vee}$ to
$t_{(i+1+b)\alpha^\vee}$ for $a,b\in\Z_{\ge0}$. The other cases are similar.
Let $d=(i+a)+(i+1+b)=2i+a+b+1$. We must show that
\begin{equation*}
  Z := \sum_{k=0}^d (-1)^k \binom{d}{k} \psi^{2i}_{-2i-2-2b+2k} \in (1-x)^d \Z[x^\pm].
\end{equation*}
Since $\psi^{2i}_{2p}=0$ for $-2i-2<2p<2i$,
\begin{align*}
  Z &= \left(\sum_{k=0}^b + \sum_{k=2i+1+b}^d\right) (-1)^k \binom{d}{k} \psi^{2i}_{-2i-2-2b+2k} \\
  &=\sum_{k=0}^b (-1)^k \binom{d}{k} \psi^{2i}_{-2i-2-2b+2k}
   + \sum_{k=0}^a (-1)^{k+2i+1+b} \binom{d}{2i+1+b+k} \psi^{2i}_{2i+2k} \\
  &=(-1)^b \sum_{k=0}^b (-1)^k \binom{d}{b-k} \psi^{2i}_{-2i-2-2k}
   - (-1)^b \sum_{k=0}^a (-1)^k \binom{d}{a-k} \psi^{2i}_{2i+2k}.
\end{align*}
Inserting the definitions~\eqref{E:psi explicit1}, \eqref{E:psi explicit2}, and~\eqref{E:Sle} we obtain
\begin{align*}
  (-1)^b Z &= \sum_{k=0}^b (-1)^k \binom{d}{b-k} (1-x^{-1})^{2i} \sum_{j=0}^k x^{-j} \binom{j+2i-1}{2i-1} \\
  &- \sum_{k=0}^a (-1)^k \binom{d}{a-k} (1-x)^{2i} \sum_{j=0}^k x^j \binom{j+2i-1}{2i-1}.
\end{align*}
Therefore we must show that $Z':=(-1)^b Z(1-x)^{-2i}$ is divisible by $(1-x)^{a+b+1}$.
Regarding $Z'$ as a function of $x$, we need to show that its $r$-th derivative at $x=1$ vanishes,
for $0\le r \le a+b$. This yields the identities
\begin{equation*}
  (-1)^r \sum_{k=0}^b (-1)^k \binom{d}{b-k} \sum_{j=0}^k \dfrac{(j+2i+r-1)!}{j!}
  =\sum_{k=0}^a (-1)^k \binom{d}{a-k} \sum_{j=r}^k \dfrac{(j+2i-1)!}{(j-r)!}.
\end{equation*}
Shifting the sums on the right hand side using $j'=j-r$ and $k'=k-r$ and dividing both sides by
$(-1)^r (2i+r-1)!$, the inner sums simplify and we obtain
\begin{equation*}
    \sum_{k=0}^b (-1)^k \binom{d}{b-k} \binom{2i+r+k}{k}
     = \sum_{k=0}^{a-r} (-1)^k \binom{d}{a-r-k} \binom{2i+r+k}{k}.
\end{equation*}
Setting $a'=a-r$, we claim that this sum is equal to $\binom{a'+b}{b} = \binom{a'+b}{a'}$, which
is symmetric in $a'$ and $b$, and hence implies the equality of both sides. This can be seen as follows.
The coefficient of $x^b$ in $(1+x)^{a'+b}$ is $\binom{a'+b}{a'}$.  Alternatively, we can calculate
\begin{equation*}
[x^b] (1+x)^{a'} (1+x)^b (1+x)^c (1+x)^{-c},
\end{equation*}
where $c = 2i+r+1$ and $(1+x)^{-c}$ is meant to be expanded as a power series in $x$. Then
\begin{equation*}
[x^b] (1+x)^{a'} (1+x)^b (1+x)^c (1+x)^{-c}
= \sum_{k=0}^b [x^{b-k}] (1+x)^{a'+b+c} [x^k](1+x)^{-c}
\end{equation*}
which is exactly the sum we wanted to evaluate.
\end{proof}

\begin{prop} \label{P:sl2smallGKM}
For all $d\ge 1$, $m\in \Z$, and $w\in W_\af$ we have
\begin{equation*}
    \psi^m((1-t_{\alpha^\vee})^{d-1}(1-r_\alpha) w) \in (1-x)^d \Z[x^\pm].
\end{equation*}
\end{prop}

\begin{proof}
Note that
\begin{equation} \label{E:rewriting}
    \psi^m((1-t_{\alpha^\vee})^{d-1}(1-r_\alpha) w)
      = \psi^m((1-t_{\alpha^\vee})^{d-1} w) - \psi^m((1-t_{\alpha^\vee})^{d-1} r_\alpha w).
\end{equation}
Furthermore, by Proposition~\ref{P:sl2 Grassmannian small GKM} $\psi^m$ satisfies
the small torus Grassmannian GKM condition.
Hence applying Lemma~\ref{L:Jcongruent} to $\psi^m((1-t_{\alpha^\vee})^{d-1} r_\alpha w)$, we
can shift the argument $r_\alpha w$ so that we can use the equalities~\eqref{E:psi explicit1}
and~\eqref{E:psi explicit2}. This implies that~\eqref{E:rewriting} is zero modulo the ideal
$(1-x)^d \Z[x^\pm]$.
\end{proof}

\subsection{Wrong-way map}
There is a natural inclusion map $\iota_I:\Psi_\af^{I} \to
\Psi_\af$. In the case at hand there is a map $\jd: \Psi_\af \to
\Psi_\af^I$, of which $\iota_I$ is a section.  This map is special
to the case of the affine Grassmannian, and also does not exist if
one uses the larger torus $T_\af$.

\begin{lem} \label{L:jd}
There is an $R(T)$-module homomorphism $\jd:\Psi_\af \to
\Psi_\af^{I}$ defined by $\jd(\psi)(w) = \psi(t_\la)$ for $w\in
W_\af$, where $\la\in Q^\vee$ is such that $w W = t_\la W$.
\end{lem}
\begin{proof} Let $\psi\in \Psi_\af$, $w=t_\la u\in W_\af$ with $\la\in Q^\vee$ and $u\in W$,
$\al\in \Phi$, and $d\in \Z_{>0}$. We have
\begin{align*}
  \jd(\psi)((1-t_{\al^\vee})^d w) &= \jd(\psi)((1-t_{\al^\vee})^d t_\la u) \\
  &= \psi((1-t_{\al^\vee})^d t_\la) \in (1-e^\al)^d R(T).
\end{align*}
\end{proof}

\section{$K$-homology of affine Grassmannian and $K$-Peterson subalgebra}
\label{S:Palgebra} Let $\K'$ be the $K$-NilHecke ring for the affine
Lie algebra $\geh_\af$ defined via the general construction of
Section \ref{S:KK}. In this section we use the {\it affine
$K$-NilHecke ring} $\K$ which differs from $\K'$ in the use of
$R(T)$ instead of $R(T_\af)$. Our main result, generalizing work of
Peterson \cite{P}, gives a Hopf-isomorphism of $K_T(\Gr_G)$ with a
commutative subalgebra $\LL \subset \K$.

\subsection{$K$-homology of affine
Grassmannian}\label{SS:Hopfstructure} We define the equivariant
$K$-homology $K_T(\Gr_G)$ of the affine Grassmannian to be the
continuous dual $K_T(\Gr_G) = \Hom_{R(T)}(K^T(\Gr_G),R(T))$, so that
$K_T(\Gr_G)$ is a free $R(T)$-module with basis the Schubert classes
$\xi_w$ dual to $[\OO_{X^I_w}] \in K^T(\Gr_G)$.

The $K$-homology $K_T(\Gr_G)$ and $K$-cohomology $K^T(\Gr_G)$ are
equipped with dual Hopf structures, which we now explain, focusing
on $K^T(\Gr_G)$ first.  Let $K\subset G$ be the maximal compact
form, $LK$ the space of continuous loops $S^1\to K$, $\Omega K$ the
space of based loops $(S^1,1)\to (K,1)$, and $T_\R = T \cap K$.  We
denote by $K^{T_\R}(\Omega K)$ the equivariant topological
$K$-theory of $\Omega K$.  By an (unpublished) well known result of
Quillen (see \cite{HHH, PS}), the space $\Omega K$ is
(equivariantly) weak homotopy equivalent to the ind-scheme affine
Grassmannian $G(\C((t)))/G(\C[[t]])$.  Thus we have $K^{T_\R}(\Omega
K) = K^{T_\R}(G(\C((t)))/G(\C[[t]]))$, where $K^{T_\R}(G(\C((t)))/G(\C[[t]]))$ denotes
the topological $K$-theory of the topological space underlying the ind-scheme 
$G(\C((t))/G(\C[[t]])$.

The topological $K$-theory $K^{T_\R}(\Omega K) \cong
K^{T_\R}(G(\C((t)))/G(\C[[t]]))$ is studied in \cite{KK:K}, where it
is identified with the ring $\Psi^I_\af$.  More precisely, Kostant
and Kumar studied the equivariance with respect to the larger torus
$T_\af$, but the same argument as in our Theorem \ref{T:GKMPet}
gives $K^{T_\R}(G(\C((t)))/G(\C[[t]])) \cong \Psi^I_\af$.  Thus we
obtain the sequence of isomorphisms
$$
K^T(\Gr_G) \cong \Psi^I_\af \cong K^{T_\R}(G(\C((t)))/G(\C[[t]]))
\cong K^{T_\R}(\Omega K)
$$
and all the isomorphisms are compatible with restrictions to fixed
points.

The composite map $r$ given by
\begin{equation*}
  \Omega K \hookrightarrow LK \longrightarrow LK/T_\R
\end{equation*}
induces the map
\begin{equation*}
  K^{T}(X_\af) \cong K^{T_\R}(LK/T_\R) \overset{r^*}{\longrightarrow} K^{T_\R}(\Omega K) \cong
  K^{T}(\Gr_G).
\end{equation*}
One can check using a fixed point calculation, that the map $\jd$ of
Lemma \ref{L:jd} is related to $r^*$ via the isomorphisms of Theorem
\ref{T:GKMPet}.

The based loop group $\Omega K$ has a $T_\R$-equivariant
multiplication map $\Omega K \times \Omega K \to \Omega K$ given by
pointwise multiplication on $K$, and this induces the structure of a
commutative and co-commutative Hopf-algebra on $K^{T_\R}(\Omega K)
\cong K^T(\Gr_G)$. The co-commutativity of $K^{T_\R}(\Omega K)$
follows from the fact that it is a homotopy double-loop space ($K$
being already a homotopy loop space).  Via duality, we obtain a dual
Hopf-algebra structure on $K_T(\Gr_G)$.  For the next result, we
label the $T_\af$-fixed points of $\Gr_G$ by translation elements
$t_\lambda$.

\begin{lem}\label{L:Hmult}
Let $\lambda,\mu \in Q^\vee$, and consider the maps $i^*_\lambda,
i^*_\mu: K^T(\Gr_G) \to R(T)$ as elements of $K_T(\Gr_G)$.  Then in
$K_T(\Gr_G)$, we have
$$
i^*_\lambda \; i^*_\mu = i^*_{\lambda + \mu}.
$$
\end{lem}
\begin{proof}
It suffices to argue in $K^{T_\R}(\Omega K)$.  The map $i^*_\lambda
\; i^*_\mu$ is induced by the map ${\rm pt} \to \Omega K \times
\Omega K \to \Omega K$ where the image of the first map is the pair
$(t_\lambda,t_\mu) \in \Omega K \times \Omega K$ of fixed points,
and the latter map is multiplication.  Treating $t_\lambda, t_\mu:S^1
\to K$ as homomorphisms into $K$, we see that pointwise
multiplication of $t_\lambda, t_\mu$ gives $t_{\lambda + \mu}$. Thus
$i^*_\lambda \; i^*_\mu = i^*_{\lambda + \mu}$.
\end{proof}

The antipode of $K_T(\Gr_G)$ is given by $S(i^*_\lambda) =
i^*_{-\lambda}$, since the fixed points satisfy $t^{-1}_\lambda =
t_{-\lambda}$ in $\Omega K$.

\subsection{Affine $K$-NilHecke ring and $K$-Peterson subalgebra}
Let $W_\af$ act on the finite weight lattice $P$ by the
(nonfaithful) level zero action  $(ut_\la\cdot \mu)=u\cdot\mu$ for
$u\in W$, $\la\in Q^\vee$ and $\mu\in P$.

Let $\K$ be the smash product of the affine 0-Hecke ring $\K_0$ with
$R(T)$ (rather than $R(T_\af)$) using the commutation relations
\eqref{E:Tcomm}. We call this the {\it affine $K$-NilHecke ring}.
The cohomological analogue of $\K$ was studied by Peterson \cite{P}.
We have $\K = \bigoplus_{w\in W_\af} R(T) \,T_w$.

We now define the map $k: K_T(\Gr_G) \to \K$ by the formula
\begin{equation}\label{E:k}
\ip{k(\xi)}{\psi} = \ip{\xi}{\jd(\psi)}
\end{equation}
where $\psi \in \Psi_\af$, and $\jd$ is the wrong-way map of Lemma
\ref{L:jd}.  We have used Theorem \ref{T:GKMPet}(2) to obtain the
pairing on the right hand side.  By letting $\psi$ vary over
$\{\psi^v \in \Psi_\af\}$, it is clear that \eqref{E:k} defines
$k(\xi)$ uniquely in $\K$.

We define the {\it $K$-Peterson subalgebra} $\LL := Z_\K(R(T))$ of
$\K$ to be the centralizer of $R(T)$ inside $\K$.

\begin{lem}\label{L:imgk}
We have $\Img(k) = \bigoplus_{\la\in Q^\vee} Q(T) t_\la \cap \K =
\LL$.
\end{lem}
\begin{proof}
For $\lambda \in Q^\vee$, we have $\ip{i^*_\lambda}{\jd(\psi)} = \psi(t_\lambda)$,
so that $k(i^*_\lambda) = t_\lambda \in \K$. Since $i^*_\lambda$ spans
$K_T(\Gr_G)$ (over $Q(T)$), we thus have established the first equality.
For the second equality, $\bigoplus_{\la\in Q^\vee} Q(T) t_\la \cap \K \subseteq
Z_\K(R(T))$ holds because under the level zero action, $t_\la$ acts
on $P$ trivially for all $\la\in Q^\vee$.  For the other direction,
let $a=\sum_{w\in W_\af} a_w w\in Z_{\K}(R(T))$ for $a_w\in Q(T)$.
Then for all $\mu\in P$ we have
\begin{equation*}
  0 = e^\mu a - a e^\mu = \sum_{w\in W_\af} a_w(e^\mu-e^{w\mu})w.
\end{equation*}
Therefore for all $w\in W_\af$ either $a_w=0$ or $w\mu=\mu$ for all
$\mu \in P$. Taking $\mu$ to be $W$-regular, we see that the latter
only holds for $w=t_\la$ for some $\la\in Q^\vee$.
\end{proof}

The algebra $\LL$ inherits a coproduct $\Delta: \LL \to \LL
\otimes_{R(T)} \LL$ from the coproduct of $\K$.  (In Subsection
\ref{SS:coprod} the coproduct of $\K'$ is given, and it specializes
easily to a coproduct for $\K$.)  That $\Delta(\LL) \subset \LL
\otimes_{R(T)} \LL$ follows from \eqref{E:tensorWeyl} and the
quality $\LL = \bigoplus_{\la\in Q^\vee} Q(T) t_\la \cap \K$.  We
make $\LL$ into a Hopf-algebra by defining $S(t_\lambda) =
t_{-\lambda}$.

The following results generalize properties of Peterson's $j$-map in
the homology case; see \cite[Theorem 4.4]{Lam:Schub}.

\begin{thm}\label{T:kmap}
The map $k: K_T(\Gr_G) \to \LL$ is a Hopf-isomorphism.
\end{thm}
\begin{proof}
To check that a map is a Hopf-morphism it suffices to check that it
is a bialgebra morphism, since the compatibility with antipodes
follows as a consequence.

It is clear from the definition that $k$ is injective. Since $k$ is
$R(T)$-linear, to check that $k$ is compatible with the
Hopf-structure we check the product and coproduct structures on the
basis $\{t_\lambda \mid \lambda \in Q^\vee\}$.  By Lemma
\ref{L:Hmult} we have $k(i^*_{\lambda+\mu}) = k(i^*_\lambda \;
i^*_\mu) = t_\lambda \; t_\mu = t_{\lambda +\mu}$, so $k$ is an
algebra morphism.  That $k$ is a coalgebra morphism follows from an
argument similar to that of Lemma \ref{L:pairprod} and Proposition
\ref{P:struct}.  Thus $k: K_T(\Gr_G) \to \LL$ is a Hopf-isomorphism.
\end{proof}

\begin{thm}\label{T:LL}
For each $w \in \Wz$, there is a unique element $k_w\in \LL$ of the
form
\begin{equation}\label{E:jform}
k_w = T_w + \sum_{v \in W_\af \setminus \Wz} k^v_w T_v
\end{equation}
for $k^v_w\in R(T)$. Furthermore, $k_w = k(\xi_w)$ and $\LL =
\bigoplus_{w\in \Wz} R(T) \, k_w$.
\end{thm}
\begin{proof}
Since the Schubert basis $\{\xi_w \mid w \in \Wz\}$ is a
$R(T)$-basis of $K_T(\Gr_G)$, by Theorem \ref{T:kmap} setting $k_w =
k(\xi_w)$ we obtain a $R(T)$-basis of $\LL$.  By \eqref{E:k} and
$\jd(\psi^v) = \psi^v$ for $v \in \Wz$ we obtain \eqref{E:jform}.
Finally, the element $k_w \in \LL$ is unique because the set $\{T_w
\mid w \in \Wz\}$ is linearly independent.
\end{proof}

Define the $T$-equivariant $K$-homological Schubert structure constants
 $d^w_{uv}\in R(T)$ for $K_T(\Gr_G)$ by
\begin{equation}\label{E:KHomstruct}
  k_u k_v = \sum_{w\in\Wz} d^w_{uv} k_w
\end{equation}
where $u,v\in \Wz$. Since $k_u \in Z_\K(R(T))$ we have
\begin{align*}
  k_u k_v = k_u \sum_{y\in W_\af} k_v^y\, T_y
  = \sum_{y\in W_\af} k_v^y \,k_u\, T_y
  = \sum_{x,y\in W_\af} k_v^y \,k_u^x\, T_x\, T_y.
\end{align*}
Applying $\psi^w$ for $w\in \Wz$ and using
\eqref{E:jform} we have
\begin{align*}
  d^w_{uv} &= \sum_{x,y\in W_\af} k_v^y\, k_u^x \,\psi^w(T_xT_y).
\end{align*}
Since $w\in \Wz$, $\psi^w(T_xT_y)=0$ unless $y\in \Wz$. But for
$y\in \Wz$, by \eqref{E:jform}, $k_v^y=\delta_{yv}$. Therefore
\begin{equation}\label{E:dbyg}
  d^w_{uv} = \sum_{\substack{x\in W_\af \\ T_x T_v=\pm T_w}} (-1)^{\ell(w)-\ell(v)-\ell(x)} k^x_u.
\end{equation}

\section{$K$-affine Fomin-Stanley algebra and $K$-homology of affine Grassmannian}
\label{S:FS} In this section we reduce to the non-equivariant
setting.  Our main result (Theorem \ref{T:L0}) describes the
specialization at 0 of $\LL$.  We will rely on the corresponding
known statements from the cohomological setting, in particular
\cite[Proposition 5.3]{Lam:Schub}.

\subsection{$K$-affine Fomin-Stanley algebra}
Define $\phi_0:R(T) \to \Z$ by $\phi_0(e^\lambda) = 1$ and extending by linearity.
 Define $\phi_0:\K\to \K_0$ by $\phi_0(a)=\sum_{w\in W} \phi_0(a_w)
T_w$, where $a=\sum_{w\in W} a_w T_w$ with $a_w\in R(T)$.

Define the {\it $K$-affine Fomin-Stanley algebra} as
\begin{equation*}
\LL_0 = \{b \in \K_0 \mid \phi_0(bq) = \phi_0(q)b \text{ for all $q
\in R(T)$ }\} \subset \K_0.
\end{equation*}
The cohomological analogue of $\LL_0$ was defined in~\cite{Lam:Schub}.

\begin{lem}
Suppose $a \in \LL$.  Then $\phi_0(a) \in \LL_0$.
\end{lem}
\begin{proof}
$\phi_0(a e^\lambda) = \phi_0(e^\lambda a) = \phi_0(a)$.
\end{proof}

In the following we shall use notation (such as $A_w$) for the
cohomological nilHecke ring.  We refer the reader to the appendix
for a review of this notation.  Let $\lessdot$ denote the covering
relation in Bruhat order.

\begin{lem}\label{L:AKagree}
Let $v \lessdot w$ in $W_\af$.   Then for each $\lambda \in P$, we have
$$\phi_0 \ip{T_w \, e^\lambda}{\psi^v} = \phi_0 \ip{A_w \lambda}{\xi^v} = \ip{A_w \la}{\xi^v}.$$
\end{lem}
\begin{proof}
Write $v = wr_\alpha$. By \cite{Hum} there exists a length-additive factorization of the form
$w=u_1 r_i u_2$ for some $i\in I_\af$ such that
$v=u_1u_2$ and $\al =u_2^{-1}\al_i$. We have
\begin{equation*}
\phi_0 \ip{T_w\,e^\la}{\psi^v} = \phi_0 \psi^v(T_w e^\lambda)
= \phi_0 (u_1 \cdot T_i \cdot e^{u_2\cdot \la})
= \phi_0  \left(\frac{e^{r_i u_2 \la} - e^{u_2\la}}{1-e^{\alpha_i}}\right)
\end{equation*}
since $\phi_0(wq)=\phi_0(q)$ for all $w\in W_\af$ and $q\in R(T)$.
Therefore
\begin{equation*}
  \phi_0 \psi^v(T_w e^\la) = \ip{\al_i^\vee}{u_2\la} = \ip{\al^\vee}{\la} = \xi^v(A_w\,\la)
\end{equation*}
using \eqref{E:Tdef} acting on an exponential for the first equality, $W_\af$-equivariance
of $\ip{\cdot}{\cdot}$ for the second equality, and Lemma \ref{L:Acommute} for the third.
\end{proof}

\begin{lem}\label{L:Ktop}
Suppose $a = \sum_{w \in W_\af} a_w T_w \in \LL_0$, where $a_w \in
\Z$.  Let $\ell$ be maximal so that $a_w \neq 0$ for some $w$ with
$\ell(w) = \ell$.  Then $a' = \sum_{\ell(w) = \ell} a_w A_w\in\B_0$.
\end{lem}

\begin{proof}
We note that for $v \in W_\af$ with $\ell(v) = \ell - 1$, we have for each $\lambda \in P$
$$
\phi_0\, \psi^v(a(e^\lambda - 1)) = \phi_0 \,\psi^v(\sum_{\ell(w) = \ell} a_w T_w (e^\lambda - 1)) = \phi_0 \,\xi^v(a' \la)
$$
using Lemma \ref{L:AKagree}.  Since $a \in \LL_0$, we have $\phi_0(a(e^\lambda - 1)) = 0$ for all $\lambda$.  Thus $a' \in \B_0$, as claimed.
\end{proof}

\begin{thm}\label{T:L0}
We have $\LL_0 = \phi_0(\LL)$.  Furthermore, $\LL_0 =
\bigoplus_{w\in\Wz} \Z\, \phi_0(k_w)$
and $\phi_0(k_w)$ is the unique element in $\LL \cap (T_w +
\bigoplus_{v \in W\setminus \Wz} \Z\, T_v)$.
\end{thm}
\begin{proof}
For $a \in \LL$, we have $\phi_0(a e^\lambda) = \phi_0(e^\lambda a)
= \phi_0(a)$.  Thus $\phi_0(\LL) \subset \LL_0$. Now suppose that $a
= \sum_{w \in W_\af} a_w T_w \in \LL_0$. Define the support of $a$
to be the $w\in W_\af$ such that $a_w\ne0$. If the support of $a$
contains a Grassmannian element $w \in \Wz$, then $a - a_w
\phi_0(k_w)\in \LL_0$, but by Theorem \ref{T:LL} its support has
fewer Grassmannian elements than $a$.  So we may suppose that $a$
has no Grassmannian element in its support. By Lemma \ref{L:Ktop},
the element $a'$ (as defined in the Lemma) lies in $\B_0$ and has no
Grassmannian support.  By \cite[Proposition 5.3]{Lam:Schub}, we must
have $a' = 0$.  Thus $a = 0$.  We conclude that $\LL_0 =
\phi_0(\LL)$.

Since $\{ \phi_0(k_w) \mid w \in \Wz \}$ is clearly linearly
independent, it follows that they form a basis.  The last statement
follows from Theorem \ref{T:LL}.
\end{proof}

Some examples of the elements $\phi_0(k_w)$, illustrating Theorem
\ref{T:L0}, are presented in Appendix \ref{S:tables_k}.

\begin{cor}\label{C:comm}
The ring $\LL_0$ is commutative.
\end{cor}
\begin{proof}
Let $a, b \in \LL_0$.  By Theorem \ref{T:L0}, we have $a + a' \in \LL$ and $b + b' \in \LL$ for some
elements $a', b'$ satisfying $\phi_0(a') = 0 = \phi_0(b')$.   Since $\LL$ is commutative we have
$$
ab = \phi_0((a + a')(b+b')) = \phi_0((b+b')(a+a')) = ba.
$$
\end{proof}

\subsection{Structure constants}

We now consider the structure constants in $\LL_0$.
The next lemma follows from either a direct calculation, or Theorem~\ref{T:k0} below.
\begin{lem} \label{L:phi0ring} For $a,b\in \LL$, we have $\phi_0(ab) = \phi_0(a)\phi_0(b)$.
\end{lem}

Applying $\phi_0$ to \eqref{E:KHomstruct}, by Lemma \ref{L:phi0ring}
for $u,v\in \Wz$ we have
\begin{equation}
  \phi_0(k_u)\phi_0(k_v) = \sum_{w\in \Wz} \phi_0(d^w_{uv}) \phi_0(k_w).
\end{equation}
That is, $\phi_0(d^w_{uv})\in\Z$ are the structure constants for the
basis $\{\phi_0(k_v)\mid v\in \Wz\}$ of $\LL_0$.

\begin{conj}\label{CJ:signKHomStruct}

For $u,v,w\in \Wz$ and $x \in W_\af$,
\begin{align*}
(-1)^{\ell(w)-\ell(u)-\ell(v)} \phi_0(d^w_{uv})& \ge0, \\
(-1)^{\ell(x)-\ell(u)} \phi_0(k^x_{u})&\ge0.
\end{align*}
By \eqref{E:dbyg}, the second statement implies the first.
\end{conj}

The tables of $\phi_0(k_w)$ in Appendix \ref{S:tables_k} support
Conjecture \ref{CJ:signKHomStruct}.

\subsection{Non-equivariant $K$-homology}
One defines the non-equivariant $K$-cohomology $K^*(\Gr_G)$ by
considering non-equivariant coherent sheaves in the natural way. We
have $K^*(\Gr_G) = \bigoplus_{w \in \Wz} \Z \, [\OO_{X^I_w}]_0$ where
$[\OO_{X^I_w}]_0$ denotes a non-equivariant class.  The
non-equivariant $K$-homology $K_*(\Gr_G)$, defined as the continuous
$\Z$-dual to $K^*(\Gr_G)$, has Schubert basis $\{\xi^0_w \mid w \in
\Wz\}$.  We have the commutative diagram
\begin{equation*}
\begin{diagram}
   \node{K_T(\Gr_G)}\arrow{s} \arrow{e,t}{\phi_0} \node{K_*(\Gr_G)}\arrow{s} \\
  \node{\LL} \arrow{n} \arrow{e,b}{\phi_0} \node{\LL_0.} \arrow{n}
\end{diagram}
\end{equation*}

The subalgebra $\LL_0$ is a Hopf-algebra, with coproduct $\phi_0
\circ \Delta$.  The following result generalizes \cite[Theorem
5.5]{Lam:Schub} to $K$-homology.

\begin{thm}\label{T:k0}
There is a Hopf-isomorphism $k_0: K_*(\Gr_G) \longrightarrow \LL_0$
such that $k_0(\xi^0_w) = \phi_0(k_w)$.
\end{thm}

\section{Grothendieck polynomials for the affine Grassmannian}
\label{S:G} In this section we specialize to affine type
$A_{n-1}^{(1)}$ and $G=SL_n(\C)$. We first introduce elements
$\k_i\in \LL_0$ which, under a Hopf algebra isomorphism $\LL_0 \cong
\Lambda_{(n)}:=\Z[h_1,\ldots,h_{n-1}]$ between the $K$-affine
Fomin-Stanley algebra and a subspace of symmetric functions,
correspond to the homogeneous symmetric functions $h_i$. For $w\in
\Wz$, the image $\KKS_w$ of $\phi_0(k_w)$ in $\Lambda_{(n)}$, is the
$K$-theoretic $k$-Schur function $\KKS_w$ which contains the
$k$-Schur function~\cite{LLM,LM2} as highest degree term. The
symmetric functions $\KKS_w$ are related to the affine stable
Grothendieck polynomials $\{G_w\mid w\in \Wz\}$ of
\cite{Lam:affStan} by duality.

\subsection{Cyclically decreasing permutations, and the elements $\k_i$}
For $G=SL_n$, we have $I=\{1,2,\dotsc,n-1\}$ and $I_\af=\{0\}\cup
I$. For $i\in I$ we wish to compute the elements
$\phi_0(k_{\sigma_i})\in\LL_0$ where $\sigma_i=r_{i-1}r_{i-2}\dotsm
r_1r_0\in W_\af$.

A {\it cyclically decreasing} element $w\in W_\af$ is one that has a
reduced decomposition $w=r_{i_1} r_{i_2} \dotsm r_{i_N}$ such that
the indices $i_1,\dotsc,i_N\in I_\af$ are all distinct, and a
reflection $r_i$ never occurs somewhere to the left of a reflection
$r_{i+1}$ where $I_\af$ is identified with $\Z/n\Z$ (so indices are
computed mod $n$). One may show that $w$ is cyclically decreasing if
and only if all of its reduced decompositions have the above
property.  Since no noncommuting braid relations can occur, all the
reduced words of $w$ also have the same indices $i_1,\dotsc,i_N$.

For $i\in I$ let $\k_i\in \K_0$ be defined by
\begin{equation} \label{E:ncgdef}
\k_i = \sum_w T_w
\end{equation}
where $w$ runs over the cyclically decreasing elements of $W_\af$ of
length $i$.  We set $\k_0 = 1$.  These elements were considered in
\cite{Lam:affStan}.

We define coordinates for the weight lattice $P$ of
$\mathfrak{sl}_n$. Let $P\subset \Z^n=\bigoplus_{i=1}^n \Z e_i$ with
fundamental weights $\omega_i = e_1+e_2+\dotsm+e_i$ and $\al_i =
e_i-e_{i+1}$ for $i\in I$.  For a subset $J \subset \{1,\ldots,n\}$,
let us write $e_J = \sum_{i \in J} e_i \in P$ for the $01$-vector
with 1's in positions corresponding to $J$.  The $e_J$ with $|J| =
k$ form the set of weights for the $k$-th fundamental representation
of $SL_n(\C)$ with highest weight $\omega_k$, which is
multiplicity-free.  We have $r_i \cdot e_J = e_{r_i \cdot J}$ where
indices are taken mod $n$.

\begin{lem} \label{lem:T}
We have
$$
T_i \cdot e^{e_J} = \begin{cases} 0 & \mbox{both $i, i+1 \in J$ or both $i, i+1 \notin J$,} \\
e^{e_{r_i \cdot J}} & \mbox{$i \in J$ and $i+1 \notin J$,} \\
-e^{e_J} & \mbox{$i \notin J$ and $i+1 \in J$.}
\end{cases}
$$
\end{lem}

Let $J, K$ be disjoint subsets of $\Z/n\Z$ such that $J \cup K \neq \Z/n\Z$.  We write $S_{J,K} \cdot e^\lambda$ for
the action of $\{r_j \mid j \in J\}$ and $\{T_k \mid k \in K\}$ on $e^\lambda$, where the operators act in the cyclically decreasing order (for example, $r_1$ would act before $T_2$).

\begin{lem} \label{L:SJK} Let $J, K$ be as above.
\begin{enumerate}
\item
If $|[0,k-1] \cap K| \geq 2$ then $S_{J,K} \cdot e^{\omega_k} = 0$.
\item
Suppose $|[0,k-1] \cap K| = 1$.  Let $a \in [0,k-1] \cap K$.  Then $S_{J,K} \cdot e^{\omega_k} = 0$
unless $[0,a-1] \subset J$.
\item
Suppose $|[0,k-1] \cap K| = 1$.  Then $S_{J,K} \cdot e^{\omega_k} = 0$ if $[k,-1] \subset (J \cup K)$.
\item
Suppose that $S_{J,K} \cdot e^{\omega_k} \neq 0$ and $[k,-1] \cap K \neq \emptyset$.
For each $a \in ([k,-1] \cap K)$, we have $[k,a] \subset (J \cup K)$.
\item
Suppose that $S_{J,K} \cdot e^{\omega_k} \neq 0$ and $[k,-1] \subset (J \cup K)$.  Then
$[0,k-1]\cap K = \emptyset$.
\end{enumerate}
\end{lem}

The following Lemma is general (not just type $A_{n-1}^{(1)}$).

\begin{lem}\label{L:L0add}
Suppose $a \in \K_0$.  If $\lambda, \mu \in P$ are such that
$\phi_0(ae^\lambda) = a$ and $\phi_0(ae^\mu) = a$, then
$\phi_0(ae^{\lambda+\mu}) = a$.
\end{lem}
\begin{proof}
Write $a (e^\la-1) = \sum_w a_w T_w$ and $T_w e^\mu = \sum_v
b_{w}^{v,\mu} T_v$ for $a_w,b_{w}^{v,\mu} \in R(T)$. We have
$\phi_0(a_w)=0$ for all $w$ and $\phi_0(a(e^\mu-1))=0$. Then
\begin{equation*}
\begin{split}
  \phi_0(a (e^{\la+\mu}-1)) &= \phi_0(a(e^\la-1)e^\mu)+\phi_0(a(e^\mu-1)) \\
  &= \phi_0(\sum_w a_w T_w e^\mu) \\
  &= \sum_v \phi_0(\sum_w a_w b_{w}^{v,\mu}) T_v = 0
\end{split}
\end{equation*}
since $\phi_0:R(T)\to \Z$ is a ring homomorphism.
\end{proof}

\begin{prop}\label{P:ki}
We have $\k_i \in \LL_0$.
\end{prop}
\begin{proof}
By Lemma \ref{L:L0add}, it is enough to prove $\phi_0(\k_i
e^\lambda) = \k_i$ for $\lambda$ either a fundamental weight or
negative of a fundamental weight.  We deal with the case that
$\lambda = \omega_k$, as negative fundamental weights are similar.

When $\phi_0(\k_i (e^{\omega_k}-1))$ is expanded in the $T_w$ basis, only terms
involving cyclically decreasing $w$ occur with non-zero coefficient.
Fix $J$.  Let us show that $[T_J] \phi_0(\k_i (e^{\omega_k}-1)) =
0$, where $T_J$ is the product of $T_j$ with $j\in J$ in cyclically decreasing order, and
$[T_J]a$ denotes the coefficient of $T_J$ in $a \in \K_0$.
This is clear if $|J| = i$, by \eqref{E:Tcomm} and Lemma \ref{lem:T}.  So suppose $|J| < i$.  Then
\begin{equation}\label{E:signsum}
[T_J] \phi_0(\k_i (e^{\omega_k}-1)) = \sum_{K\,:\, |K| = i-|J| \;
\text{and} \; K \cap J = \emptyset} \phi_0(S_{J,K}
e^{\omega_k}).\end{equation} Let us call a subset $K$ in the above
summation {\it good}, if $S_{J,K} e^{\omega_k} \neq 0$.  Let us
define an involution $\iota$ on good subsets so that $\phi_0(S_{J,K}
e^{\omega_k}) = - \phi_0(S_{J,\iota(K)} e^{\omega_k})$.  By Lemma
\ref{L:SJK}(1), we may write the set of good subsets as the disjoint
union $S_0 \sqcup S_1$ where
$$
 S_0 = \{K \mid |K \cap [0,k-1]| = 0\} \quad \text{and} \quad S_1 = \{K \mid |K \cap [0,k-1]| = 1\}.
$$
The involution will satisfy $\iota(S_a) = S_{1-a}$.

Suppose that $K \in S_0$.  Then $K \cap [k,-1] \neq \emptyset$, and
we may set $a$ to be the maximal element of $K \cap [k,-1]$.  We let
$\iota(K) = K \backslash \{a\} \cup \{b\}$ where $b \in [0,k-1]$ is
minimal so that $j \notin J$ and $[0,j-1] \subset J$.  One checks
directly that $\iota(K)$ is also good.  Using Lemma \ref{lem:T}, one
sees that $\iota(K)$ and $K$ contribute different signs in
\eqref{E:signsum}.

Suppose that $K \in S_1$.  Let $a$ be the unique element in $K \cap
[0,k-1]$.  Using Lemma \ref{L:SJK}(3) we pick $b \in [k,-1]
\backslash (J \cup K)$ minimal in $[k,-1]$.  We set $\iota(K) = K
\backslash \{a\} \cup \{b\}$.  Again one checks that $\iota(K)$ is
good and that $K$ and $\iota(K)$ contribute different signs.

Finally, it follows from Lemma \ref{L:SJK}(2,4) that $\iota$ is an
involution.
\end{proof}

\begin{cor}\label{C:gi}
For $1\le i\le n-1$, $\k_i = \phi_0(k_{\sigma_i})$.
\end{cor}
\begin{proof}
By Proposition \ref{P:ki}, $\k_i\in \LL_0$ and by definition has
unique Grassmannian term $T_{\sigma_i}$. The result follows by
Theorem \ref{T:L0}.
\end{proof}

Forgetting equivariance we obtain a $K$-homology Pieri rule for
$K_*(\Gr_{SL_n})$.  See \cite{LLMS, LM1} for the homological
version.
\begin{cor}\label{C:Pieri} For $1\le i\le n-1$,
$\phi_0(d^w_{\sigma_i,v})$ equals $(-1)^{\ell(w)-\ell(v)-i}$ times the number of
cyclically decreasing elements $x\in W_\af$ with $\ell(x)=i$ and $T_xT_v=\pm T_w$.
\end{cor}
\begin{proof} This follows from \eqref{E:dbyg}, \eqref{E:ncgdef}, and Corollary \ref{C:gi}.
\end{proof}

The $K$-cohomology Pieri rule is likely to be much more complicated;
see \cite{LLMS} for the cohomological version.

\subsection{Coproduct of the $\k_i$'s}
In this section we determine the coproduct $\phi_0(\Delta(\k_i))$
explicitly.


Let $J$ and $K$ be two subsets of $\Z/n\Z$ with total size less than
$n-1$.   We define a sequence of nonnegative integers $\cd_{J,K} =
(\cd(j): j \in \Z/n\Z)$ by
$$
\cd(j) = \max_{0 \leq t \leq n-1} \{ |J \cap [j-t,j)| + |K \cap [j-t,j)|-t\}.
$$
(The intervals $[j-t,j)$ are to be considered as cyclic intervals.)
 Then it is clear that $\cd(j+1) - \cd(j) \in \{-1,0,1\}$.

We note that $\cd(j)\ge0$ for all $j\in\Z/n\Z$.

\begin{lem}\label{L:cdprop}
Let $J$ and $K$ be two subsets of $\Z/n\Z$ with total size less than $n-1$.
\begin{enumerate}
\item There exists $j$ so that $\cd(j) = 0$ and $j \notin J \cup K$.
\item $\cd$ is the unique sequence such that $\cd(j+1) - \cd(j) = |j \cap J| + |j \cap K| - 1$, except when
$\cd(j) = 0$ and $j \notin (J \cup K)$.
\end{enumerate}
\end{lem}
\begin{proof}
We prove (1).  Suppose that no such $j$ exists.  Then $\cd(j+1) - \cd(j) = |j \cap J| + |j \cap K|-1$ for each $j$. But $0 = (\cd(j) - \cd(j-1)) + \cdots + (\cd(j+1)-\cd(j))$, so this is impossible because $|J| + |K| \leq n-1$.
Now we prove (2).  Everything except uniqueness is clear.  Let $\cd'$ be any sequence with these properties.  The same calculation as in (1) shows that there is $j'$ so that $\cd'(j') = 0 = \cd'(j'+1)$ and $j' \notin (J \cup K)$.  It follows from recursively calculating $\cd(j'+1)$ and $\cd'(j'+1)$, then $\cd(j'+2)$ and $\cd'(j'+2)$, and so on, that $\cd(j) \geq \cd'(j)$ for all $j$.  But a symmetric argument shows that $\cd(j) \leq \cd'(j)$ for all $j$.
\end{proof}

Define $t(J,K) = (t_i: i \in \Z/n\Z) \in \{L,R,B,E\}^{n}$ as follows (where $E = $ empty, $L = $ left, $R =$ right, and $B = $ both):
$$
t_j = \begin{cases} E & \mbox{if $\cd(j) = 0$ and $j \notin J \cup K$,} \\
L & \mbox{if $\cd(j) = 0$ and $j \in J \backslash K$,} \\
R & \mbox{if $j \notin J$ and ($\cd(j) > 0$ or $j \in K$),}\\
B & \mbox{otherwise.}
\end{cases}
$$
Say that two sequences $\cd$ and $t$ are {\it compatible} if (a)
$t_j \in \{E,L\}$ implies $\cd(j) = 0$, (b) $\cd(j+1) - \cd(j) = 0$
if $t_j = L$, (c) $\cd(j+1) - \cd(j) \in \{-1,0\}$ if $t_j = R$, (d)
$\cd(j+1) - \cd(j) \in \{0,1\}$ if $t_j = B$, and finally (e) $t_j =
E$ for some $j \in [0,n-1]$.  Define the support of $(\cd,t)$ to be
$\{j \mid t_j \neq E\}$.

\begin{lem}\label{L:compatiblebiject}
The map $(J,K) \mapsto (\cd_{J,K}, t(J,K))$ is a bijection between
pairs of subsets of $\Z/n\Z$ with total size $k < n$, and pairs of
compatible sequences with support of size $k$.
\end{lem}
\begin{proof}
It is easy to see that $(\cd_{J,K}, t(J,K))$ is compatible with
support of the correct size.  We first check that the pair of
sequences determines $J ,K$.  By itself, $t(J,K)$ completely
determines $J$: we have $j \in J$ if and only if $t_j \in \{L,B\}$.
Also $j \in K$ if and only if either ($t_j = R$ and $\cd(j+1) =
\cd(j)$) or ($t_j = B$ and $\cd(j+1) = \cd(j) + 1$).  Thus
$(\cd_{J,K}, t(J,K))$ determine $J,K$.

Conversely, given compatible $(\cd,t)$, we recursively construct $J,
K$ by starting at some value $j$ such that $t_j = E$.  For such a
value we have $j \notin J \cup K$.  Then we decide whether $j+1 \in
J$ and/or $j +1 \in K$, and so on.  By construction, we obtain two
subsets $J,K$ such that $\cd(j+1) - \cd(j) = |j \cap J| + |j \cap
K|-1$, unless $\cd(j) = 0 = \cd(j+1)$ and $j \notin J \cup K$.  By
Lemma \ref{L:cdprop}, we have $\cd_{J,K} = \cd$.  Using
compatibility, one checks that the size of the support of $(\cd,t)$
is equal to $|J| + |K|$.  But then it follows that $t(J,K) = t$.
\end{proof}

\begin{prop}\label{P:kcoprod}
We have $\phi_0(\Delta(\k_r)) = \sum_{0 \leq j \leq r} \k_j \otimes
\k_{r-j}$.
\end{prop}
\begin{proof}  Our proof follows the strategy in \cite[Section 7.2]{Lam:Schub}.
Let $J = \{i_1,\ldots, i_r\} \subset \Z/n\Z$.  We use \eqref{E:DeltaT}.
We calculate $\phi_0(\Delta(T_J))$ by expanding
$$
D = \phi_0(\Delta(T_{i_1}) \cdot \Delta(T_{i_2}) \cdots
\Delta(T_{i_\ell}))
$$
where $r_{i_1} \cdots r_{i_\ell}$ is a cyclically decreasing reduced
expression and $\cdot$ means the ``componentwise" product on
$\Delta(\K)$ of \eqref{E:Deltamult}.

Let us expand this product, by picking one of the three terms in
each parentheses of \eqref{E:DeltaT}.  As usual, we write $\alpha_{ij} = \alpha_i +
\cdots + \alpha_{j-1}$ for any cyclic interval $[i,j]$.  We first
note that
\begin{align*}
T_j(1 - e^{\alpha_{i,j}}) &= (1-e^{\alpha_{i,j+1}})T_j + \frac{1 - e^{\alpha_{i,j+1} }- (1 - e^{\alpha_{i,j}})}{1-e^{\alpha_j}} \\
&=(1-e^{\alpha_{i,j+1}})T_j  + e^{\alpha_{i,j}}.
\end{align*}
Because of the cyclically decreasing condition, whenever the above
calculation is encountered, the coefficient $e^{\alpha_{i,j}}$ will
always commute with any $T_i$ which occur to the left.

We shall now show by induction that the only terms in the expansion
of $\Delta(T_{i_k}) \cdots \Delta(T_{i_\ell})$ which contribute to
$D$ look like
\begin{equation}\label{E:form}
T_v \otimes \prod_{i \in S} (1 - e^{\alpha_{i,i_k+1}})q T_w
\end{equation}
where
\begin{enumerate}
\item[(i)]
either $S$ is empty and $\prod_{i \in S} (1 - e^{\alpha_{i,i_k+1}})
= 1$, or $S \subset \{i_k,i_{k+1},\ldots,i_\ell\}$;
\item[(ii)]
$q \in R(T)$ commutes with $r_{i_1}, \ldots, r_{i_{k-1}}$ and satisfies $\phi_0(q) = 1$.
\end{enumerate}
Such a term contributes nothing to $D$ if $|S| > 0$ and $i_k \notin
\{i_1,\ldots,i_{k-1}\}$.  To prove the inductive step we may assume
that $i_{k-1} = i_k+1$ and make the calculation (using
\eqref{E:deriv})
\begin{equation} \label{E:split}
T_{i_k+1} \prod_{i \in S} (1 - e^{\alpha_{i,i_k+1}})q  = \prod_{i
\in S} (1 - e^{\alpha_{i,i_k+2}})q T_{i_k+1} + \prod_{i \in S'}(1 -
e^{\alpha_{i,i_k+2}}) + q'
\end{equation}
where $S' \subset S$, and $q' \in R(T)$ commutes with $r_{i_1},
\ldots, r_{i_{k-2}}$ and satisfies $\phi_0(q')= 0$.  Clearly the
term involving $q'$ contributes nothing to $D$, and the first two
terms lead to expressions of the form \eqref{E:form}.

Given a choice of one of the three terms in each parentheses, we
define a sequence $t_j$ by (a) $t_j = E$ if $j \notin J$, (b) $t_j =
L$ if we pick $T_j \otimes 1$, (c) $t_j = R$ if we pick $1 \otimes
T_j$, and (d) $t_j = B$ if we pick $T_j \otimes
(1-e^{\alpha_j})T_j$.  Furthermore, let us make a choice of one of
the two terms in \eqref{E:split}, whenever we have such a choice.
At each step of our calculation we are looking at a term of the form
$\eqref{E:form}$.  We set $\cd(j)$ to be the size of $S$ in the term
just before $\Delta(T_j)$ is applied.  If $j \notin J$ then $\cd(j)
= 0$.  If this entire process produces a non-zero term of $D$, then
$(\cd,t)$ is a compatible sequence: the sequence $\cd$ ``wraps
around'' properly because eventually the coefficient $\prod_{i \in
S} (1 - e^{\alpha_{i,i_k+1}})$ has to equal 1, otherwise it will
vanish when $\phi_0$ is applied.  Conversely, a compatible pair
$(\cd,t)$ with support equal to $J$ always arises in this fashion.

By Lemma \ref{L:compatiblebiject}, there is a bijection between
compatible pairs $(\cd, t)$ with support of size $r$ and pairs of
subsets $J, K$ with total size equal to $r$.  It is easy to check
that the term in $D$ corresponding to $(\cd,t)$ is exactly $T_J
\otimes T_K$.
\end{proof}

\subsection{Symmetric function realizations}
Let $\La=\bigoplus_\la \Z m_\la$ be the ring of symmetric functions
over $\Z$, where $m_\la$ is the monomial symmetric function
\cite{Mac} and $\la = (\la_1 \geq \la_2 \geq \cdots \geq \la_\ell >
0)$ runs over all partitions. Let $\hat{\La}= \prod_\la \Z m_\la$ be
the graded completion of $\La$.  Let $|\la| = \lambda_1 + \cdots +
\la_\ell$ denote the size of a partition.

Let $\La^{(n)} = \La/\langle m_\la \mid \la_1 \geq n \rangle$ denote
the quotient by the ideal generated by monomial symmetric functions
labeled by partitions with first part greater than $n$. We write
$\hat{\La}^{(n)}$ for the graded completion of $\La^{(n)}$. Now let
$\La_{(n)} = \Z[h_1,h_2,\ldots,h_{n-1}] \subset \La$ denote the
subalgebra generated by the first $(n-1)$ homogeneous symmetric
functions.  Both $\La_{(n)}$ and $\hat{\La}^{(n)}$ are Hopf
algebras.

The {\it Hall inner product} $\ip{.}{.}: \La \times_\Z \La \to \Z$
extends by linearity with respect to infinite graded linear
combinations to a pairing $\ip{.}{.}:\La \times_\Z \hat{\La} \to
\Z$, which in turn descends to a pairing $\La_{(n)} \times_\Z
\La^{(n)} \to \Z$.  This pairing expresses $\La_{(n)}$ as the
continuous (Hopf)-dual of $\hat{\La}^{(n)}$ and expresses
$\hat{\La}^{(n)}$ as the graded completion of the graded (Hopf)-dual
of $\La_{(n)}$.  For short we will just say that $\La_{(n)}$ and
$\hat{\La}^{(n)}$ are dual.  The basis $\{h_\lambda \mid \lambda_1 <
n\} \subset \La_{(n)}$ and ``basis'' $\{m_\lambda \mid \lambda_1 <
n\} \subset \hat{\La}^{(n)}$ are dual under the Hall inner product.

\subsection{Affine stable Grothendieck polynomials}
The {\it affine stable Grothendieck polynomials}
$G_v(x_1,x_2,\ldots)$ for $v \in W_\af$ are the formal power series
defined by the identity \cite{Lam:affStan}
\begin{align}\label{E:Gdef}
  \prod_{i\ge 1} \sum_{j=0}^{n-1} (x_i^j \k_j) = \sum_{v\in W_\af} G_v(x_1,x_2,\ldots) T_v
\end{align}
where the $x_i$ are indeterminates that commute with the elements of
$\K_0$ and $\k_j\in \K_0$ are the elements defined in
\eqref{E:ncgdef}. Alternatively, for a composition $\alpha =
(\alpha_1,\alpha_2,\ldots, \alpha_\ell)$, the coefficient of
$x^\alpha = x_1^{\alpha_1} x_2^{\alpha_2} \cdots
x_\ell^{\alpha_\ell}$ in $G_v(x)$ is equal to the coefficient of
$T_v$ in $\k_{\alpha_1}\k_{\alpha_2} \cdots \k_{\alpha_\ell}$.  It
is clear that $G_v(x)$ is a sum of monomials such that no variable
occurs with degree more than $n-1$ in any monomial.  Examples of the
$G_v(x)$ are given in Appendix \ref{S:table G}.

The following result was proved directly in \cite[Theorem
44]{Lam:affStan}.

\begin{prop}
For each $v \in W_\af$, we have $G_v(x) \in \hat{\La}^{(n)}$.
\end{prop}
\begin{proof}
By Proposition \ref{P:ki} and Corollary \ref{C:comm}, the $\k_i$
commute.  This implies that $G_v(x)$ is a symmetric function.  In addition,
all monomial symmetric functions $m_\la$ which occur with non-zero coefficient in $G_v(x)$
satisfy $\la_1 < n$, so that $G_v(x)$ can be naturally identified with its image in $\hat{\La}^{(n)}$.
\end{proof}

The next result follows from \eqref{E:Tbraid}.
\begin{lem}
The graded components of $G_v(x)$ are alternating.  That is, the
coefficient of $m_\lambda$ in $G_v(x)$ has sign equal to that of
$(-1)^{|\lambda| - \ell(v)}$.
\end{lem}

In \cite{Lam:affStan}, for each $v \in W_\af$, a homogeneous
symmetric function $F_v(x)$, the {\it affine Stanley symmetric
function}, is defined.  The next result follows by inspection.

\begin{lem}\label{L:Gdeg}
Let $v \in \Wz$.  The lowest degree component (degree $\ell(v)$) of
$G_v(x)$ is equal to $F_v(x)$.
\end{lem}

(Readers not familiar with affine Stanley symmetric functions may
take this as the definition of $F_v(x)$.)  The set $\{F_v(x) \mid v \in
\Wz\}$ were called {\it affine Schur functions} in
\cite{Lam:affStan}, and are equal to the dual $k$-Schur functions of
\cite{LM2}.

\begin{prop}\label{P:Gbasis}
The set $\{G_v(x) \mid v \in \Wz\}$ is a ``basis'' of
$\hat{\La}^{(n)}$.  In other words, $\hat{\La}^{(n)} = \prod_{v \in
\Wz} \Z\, G_v(x)$.
\end{prop}
\begin{proof}
This follows from the fact that $\{F_v(x) \mid v \in \Wz\}$ is a basis of
$\La^{(n)}$ \cite{Lam:affStan, LM2}.
\end{proof}

\begin{rem}\label{R:stable}
Suppose $w \in W_\af$ is such that some (equivalently, every)
reduced expression for $w$ does not involve all the simple
generators $r_0,r_1,\ldots,r_{n-1}$.  Then it follows from comparing
the definitions that the stable affine Grothendieck polynomial
$G_w(x)$ is equal to the usual stable Grothendieck polynomial
\cite{FK} labeled by $u \in W = S_n$, where $u$ is obtained from $w$
by cyclically rotating the indices until $r_0$ is not present.
\end{rem}

\begin{rem}\label{R:relabel}
There is a bijection between $v \in \Wz$ and $(n-1)$-bounded
partitions $\{\la \mid \la_1 < n\}$ \cite{LM1, Lam:affStan}.  The partition $\la$ associated
to $v$ can be obtained from the exponents of the dominant monomial term
$x_1^{\la_1}x_2^{\la_2} \cdots$ in $F_v(x_1,x_2,\ldots)$. We may
thus relabel $\{G_v(x) \mid v \in \Wz\}$ as $\{G_\lambda^{(k)}(x)
\mid \la_1 < n\}$, where $k = n-1$ (due to historical reasons).
A table with this correspondence is given in Appendix~\ref{S:grass-kbounded}.
Remark \ref{R:stable} implies that $G_\lambda^{(k)}(x) =
G_\lambda(x)$ whenever the largest hook length of $\lambda$ is less
than or equal to $k$, where $G_\lambda(x)$ is the stable
Grothendieck polynomial labeled by partitions, studied by Buch~\cite{B}.
\end{rem}

\subsection{$K$-theoretic $k$-Schur functions}
Since $\{G_v(x) \mid v \in \Wz\}$ is a ``basis'' of
$\hat{\La}^{(n)}$, there is a dual basis $\{g_v(x) \mid v \in \Wz\}$
of $\La_{(n)}$.  This definition of $g_v(x)$ was made previously by
Lam. We call the symmetric functions $g_v(x)$ {\it affine dual
stable Grothendieck polynomials}, or {$K$-theoretic $k$-Schur
functions}.  Examples of the $g_v(x)$ are given in Appendix~\ref{S:tables g}.

The proof of the following result is standard (see for example
\cite[Lemma 7.9.2]{EC2}).

\begin{lem}\label{L:Cauchy}
We have
$$
\sum_{\lambda_1 < n} h_\la(x)\,m_\la(y) = \sum_{v \in \Wz} g_v(x) \,
G_v(y).
$$
\end{lem}

Let $k = n-1$.  The {\it $k$-Schur functions} $\{s^{(k)}_v(x) \mid v
\in \Wz\}$ \cite{LLM,LM2} are the basis of $\La_{(n)}$ dual to
$\{F_v(x) \mid v \in \Wz\}$, and are usually labeled by the
$k$-bounded partitions $\{\la \mid \la_1 < n\}$.  (The $k$-Schur
functions originally defined in \cite{LLM} depend on an additional
parameter $t$, and setting $t = 1$ conjecturally gives the $k$-Schur
functions of \cite{LM2}, which are the ones used here.)

\begin{lem} \label{L:ghighestComp}
Let $v \in \Wz$.  Then the highest degree homogeneous component of
$g_v(x)$ is equal to the $k$-Schur function $s^{(k)}_v(x)$.
\end{lem}
\begin{proof}
We prove this by induction on $\ell(v)$.  The base case is clear:
$G_{\id}(x) = F_{\id}(x) = 1 = g_{\id}(x) = s^{(k)}_\id(x)$. Suppose
the claim has been proven for all $w$ satisfying $\ell(w) < \ell$,
and let $\ell(v) = \ell$.  One then checks that the symmetric
function
$$
g_v(x) = s^{(k)}_v(x) - \sum_{w:\; \ell(w) < \ell}
\ip{s^{(k)}_v(x)}{G_w(x)}g_{w}(x)
$$
is a solution to the system of equations
$$
\ip{g_v(x)}{G_u(x)} = \delta_{vu} \ \ \ \mbox{for all $u \in \Wz$.}
$$
\end{proof}

\begin{rem}
Relabel $\{g_v(x)\mid v \in \Wz\}$ as $\{g^{(k)}_\la(x) \mid \la_1
\leq k\}$ as in Remark \ref{R:relabel}.  Since $\lim_{k \to \infty}
G^{(k)}_\la(x) = G_\la(x)$, it follows that $\lim_{k \to \infty}
g^{(k)}_\la(x) = g_\la(x)$, where $g_\la(x)$ are the dual affine
stable Grothendieck polynomials studied in \cite{L,LP}.
\end{rem}

\subsection{Noncommutative $K$-theoretic $k$-Schur functions}
Define $\varphi: \La_{(n)} \to \LL_0$ by $h_i \mapsto \k_i$. This
map is well-defined since the $h_i$ are algebraically independent,
and $\LL_0$ is commutative.  The {\it noncommutative $K$-theoretic
$k$-Schur functions} are the elements $\{\varphi(g_v) \mid v\in
\Wz\} \subset \LL_0$.

\begin{prop}\label{P:noncomm}
Let $w \in W_\af$ and $v \in \Wz$.  The coefficient of $T_w$ in
$\varphi(g_v)$ is equal to the coefficient of $G_v(x)$ in $G_w(x)$,
when the latter is expanded in terms of $\{G_u(x) \mid u \in \Wz\}$.
\end{prop}
\begin{proof}
Applying $\varphi$ to Lemma \ref{L:Cauchy}, and comparing with
\eqref{E:Gdef}, we have $$\sum_{v \in \Wz} G_v(y) \varphi(g_v) =
\sum_{w \in W_\af} G_w(y) \, T_w.$$  Now take the coefficient of
$T_w$ on both sides.
\end{proof}

\subsection{Grothendieck polynomials for the affine Grassmannian}
This is our main theorem.
\begin{thm} \label{T:symmfunc} \mbox{}
\begin{enumerate}
\item The map $\varphi: \La_{(n)} \to \LL_0$ is a Hopf isomorphism,
sending $g_v$ to $\phi_0(k_v)$ for $v \in \Wz$.
\item We have a Hopf algebra isomorphism $k_0^{-1} \circ \varphi:
\La_{(n)} \to K_*(\Gr_{SL_n})$, sending $g_v$ to $\xi^0_v$ for $v
\in \Wz$.
\item There is a dual Hopf algebra isomorphism
$K^*(\Gr_{SL_n}) \cong \hat{\La}^{(n)}$, sending $[\OO_{X^I_v}]_0$
to $G_v(x)$ for $v \in \Wz$.
\item The following diagram commutes:
\begin{equation*}
\begin{diagram}
\node{K_*(\Gr_{SL_n}) \times K^*(\Gr_{SL_n})} \arrow{e}\arrow{s,<>} \node{\Z} \arrow{s,b}{\text{{\rm id}}} \\
\node{\La_{(n)} \times \hat{\La}^{(n)}} \arrow{e} \node{\Z\;.}
\end{diagram}
\end{equation*}
\end{enumerate}
\end{thm}
\begin{proof}
Given the definitions and Theorem \ref{T:k0}, all the statements
follow from the first one.  By Theorem \ref{T:L0}, and Propositions
\ref{P:Gbasis} and \ref{P:noncomm} we deduce that $\varphi(g_v) =
\phi_0(k_v)$.  It follows that $\varphi$ is an isomorphism.  Since
$\Delta(h_i) = \sum_{0 \leq j \leq i} h_j \otimes h_{i-j}$ in
$\La_{(n)}$, it follows from Proposition \ref{P:kcoprod} that
$\varphi$ is a Hopf-morphism.
\end{proof}

\begin{cor}
For $1 \leq r \leq n -1$, we have $g_{\sigma_r}(x) = h_r(x)$.
\end{cor}

Recall the map $r^*:K^T(X_\af) \to K^T(\Gr_G)$ defined in Subsection
\ref{SS:Hopfstructure}.  We use $r^*_0: K^*(X_\af) \to
K^*(\Gr_{SL_n})$ to denote the evaluation of $r^*$ at 0.

\begin{thm}\label{T:ASG}
The image of $G_w(x)$ under the isomorphism $\hat{\La}^{(n)} \cong
K^*(\Gr_{SL_n})$ is equal to the $r_0^*([\OO_{X_w}]_0)$.
\end{thm}
\begin{proof}
As observed previously, the map $\jd$ of Lemma \ref{L:jd} is related
to $r^*$ via the isomorphisms of Theorem \ref{T:GKMPet}.  By
\eqref{E:k}, $\ip{k(\xi_v)}{\psi^w} = \ip{\xi_v}{\jd(\psi^w)}$. It
follows that the coefficient of $T_w$ in $k_v$ is equal to the
coefficient of $[\OO_{X^I_v}]$ in $r^*([\OO_{X_w}])$. Applying
$\phi_0$ to these coefficients and comparing with Proposition
\ref{P:noncomm} gives the result.
\end{proof}

\subsection{Conjectural properties}
In this section we list conjectural properties of the symmetric
functions $g_w(x)$ and $G_w(x)$.  When $w \in \Wz$, we will use
partitions to label these symmetric functions, see Remark
\ref{R:relabel}.  Recall also that $k = n-1$.


\begin{conj} \label{C:g}
The basis $g^{(k)}_\lambda$ of $\La_{(n)}$ satisfies
\begin{enumerate}
\item
$g^{(k)}_\lambda$ is a positive integer (necessarily finite) sum of $k$-Schur functions
(by Lemma~\ref{L:ghighestComp} the top homogeneous component of $g^{(k)}_\lambda$
is the $k$-Schur function $s^{(k)}_\lambda$).
\item
The coproduct structure constants $c^{\mu\nu}_\la$ in
$\Delta(g^{(k)}_\la)= \sum_{\mu,\nu} c^{\mu\nu}_\la
g^{(k)}_\mu\otimes g^{(k)}_\nu$ are alternating integers, that is
$(-1)^{|\la|-|\nu|-|\mu|} c^{\mu\nu}_\la \in \Z_{\ge 0}$.
Furthermore $c_\la^{\mu\nu}=0$ unless $|\mu|+|\nu|\le |\la|$.
\item
The coefficients in the expansion $g^{(k)}_\lambda=\sum_\mu a_\la^\mu \;g^{(k+1)}_\mu$
are alternating integers, that is $(-1)^{|\la|-|\mu|} a_\la^\mu\in \Z_{\ge 0}$.
\end{enumerate}
\end{conj}

Conjecture~\ref{C:g} (1) has been checked for $n=2,3,4,5$ for
$|\la|\le 8$ using Sage~\cite{Sage}, see also the tables in
Appendix~\ref{S:tables g}. Data confirming Conjecture~\ref{C:g} (2)
can be found in Appendix~\ref{S:coproduct g}. Conjecture~\ref{C:g}
(3) has been checked for $n=2,3,4$ for $|\la|\le 8$ using Sage.
According to Conjecture~\ref{CJ:signKHomStruct}, the product
structure constants for $\{g^{(k)}_\lambda\}$ should be alternating.

\begin{conj}\label{C:G}\mbox{}
\begin{enumerate}
\item
Every affine stable Grothendieck polynomial $G_w$ for $w \in W_\af$ is a {\it finite}
alternating linear combination of $\{G^{(k)}_\lambda\}$.
\item
Every $G^{(k)}_\lambda$ is an {\it alternating} integer linear combination of the
affine Schur functions $\{F^{(k)}_\mu\}$.
\item
The structure constants in the product $G^{(k)}_\mu \, G^{(k)}_\nu =
\sum_\la c^{\mu\nu}_\la \; G^{(k)}_\la$ are alternating integers,
that is $(-1)^{|\la|-|\mu|-|\nu|} c^{\mu\nu}_\la \in \Z_{\ge 0}$.
Furthermore $c_\la^{\mu\nu}=0$ unless $|\mu|+|\nu|\le |\la|$.
\item
The coefficients in the expansion $G^{(k+1)}_\mu = \sum_\lambda
a^\mu_\la \; G_\la^{(k)}$ are alternating integers, that is
$(-1)^{|\la|-|\mu|} a_\la^\mu\in \Z_{\ge 0}$.
\end{enumerate}
\end{conj}

By Proposition~\ref{P:noncomm}, the ``alternating'' part of
Conjecture~\ref{C:G} (1) is implied by
Conjecture~\ref{CJ:signKHomStruct}. Evidence for
Conjecture~\ref{C:G} (2) is provided in the table of
Appendix~\ref{S:table G}. Conjecture~\ref{C:G} (2) is related to
Conjecture~\ref{C:g} (1) by a matrix inverse. Conjecture~\ref{C:G}
(3) is equivalent to Conjecture~\ref{C:g} (2); indeed, the two sets
of structure constants are identical. Conjecture~\ref{C:G} (4) is
equivalent to Conjecture~\ref{C:g} (3).

\begin{rem}
The factorization of affine Grassmannian homology Schubert classes as described in \cite{Mag} (see also \cite{Lam:Schub,LM2}) also appears to hold in some form in $K$-homology.  Suppose $w \in W_\af^I$ has a length additive factorization $w = vu$ where $u \in W_\af^I$ is equal modulo length 0 elements, to the translation $t_{-\omega_i^\vee}$ by a negative fundamental coweight in the extended affine Weyl group (\cite{Mag}), or equivalently the partition $\la$ corresponding to $u$ is a rectangle of the form $\ell \times (k-\ell)$.  Then it appears that $g_w$ is a multiple of $g_u$ in $\La_{(n)}$.
\end{rem}

\appendix
\section{Affine NilHecke ring and Tables}
\label{S:A}

\subsection{(Cohomological) Affine NilHecke ring}
A summary of the (notational) correspondence between (co)homology
and $K$-(co)homology is given in Table~\ref{tab:term}.  Some of our notation differs with
that from~\cite{Lam:Schub}.

\begin{table}
\begin{tabular}{|l|l|l|}
    \hline
    (co)homology & $K$-(co)homology & terminology\\ \hline
    $A_i$ & $T_i$ & ($K$-)NilHecke generators\\[1mm]
    $\A = \bigoplus_w S A_w$ & $\K = \bigoplus_w R(T) T_w$ & ($K$-)NilHecke ring\\[1mm]
    $\B = Z_\A(S)$ & $\LL = Z_\K(R(T)) $ & Peterson's subalgebra\\[1mm]
    $\B_0 = \phi_0(\B)$ & $\LL_0 = \phi_0(\LL)$ & affine Fomin-Stanley subalgebra\\[1mm]
    $\{j_w\} \subset \B$ & $\{k_w\} \subset \LL$
     &Schubert basis\\[2mm]
    $s_w^{(k)}(x)$ & $g_w(x)$
     & ($K$-theoretic) $k$-Schur functions\\[1mm]
     $F_w(x)$ & $G_w(x)$ & affine Stanley symmetric functions/ \\
      &&stable affine Grothendieck polynomials\\[1mm]
    \hline
\end{tabular}
\vspace{1mm} \caption{Terminology \label{tab:term}}
\end{table}

We now recall the affine NilHecke ring $\A$. Let $S=\Sym(P)$ where
$P$ is the weight lattice of the finite-dimensional group $G$.
$W_\af$ acts on $P$ (and therefore on $S\cong H^T(\pnt)$) by the
level zero action.
The affine NilCoxeter algebra $\A_0$ is the ring with generators $\{A_i\mid i\in I_\af \}$ and
relations
\begin{equation*}
  A_i^2 = 0 \ \ \ \  \text{and} \ \ \ \
 \underbrace{A_iA_j \dotsm}_{\text{$m_{ij}$ times}} = \underbrace{A_j A_i\dotsm}_{\text{$m_{ij}$ times}.}
\end{equation*}
Define $A_w$ in the obvious way and define the nilCoxeter algebra
$\A_0 = \bigoplus_{w\in W_\af} \Z A_w$.  Then $\A_0$ acts on $S$ by
\begin{equation*}
\begin{split}
  A_i \cdot \la &= \ip{\al_i^\vee}{\la} \\
  A_i \cdot (ss') &= (r_i\cdot s) A_i\cdot s' + (A_i\cdot s) s'
\end{split}
\end{equation*}
for $i\in I_\af$, $\la\in P$, and $s,s'\in S$.

The affine Kostant-Kumar NilHecke ring $\A$  \cite{P} is the smash product of
$\A_0$ and $S$. It has relations
\begin{equation*}
  A_i s = (r_i \cdot s) A_i + (A_i\cdot s)
\end{equation*}
for $i\in I_\af$ and $s\in S$. Then
\begin{equation*}
  \A = \bigoplus_{w\in W_\af} S A_w.
\end{equation*}
In $\A$ we have $r_i=1-\al_i A_i$.

$\A$ acts on $\F(W,S)$ by
\begin{equation*}
  (a\cdot \xi)(w) = \xi(wa)
\end{equation*}
viewing $\xi\in \F(W,S)$ as an element of $\Hom_Q(\A_Q,Q)$ (left $Q$-module homomorphisms),
where  $Q=\Frac(S)$.
\begin{lem}\label{L:Acommute} \cite{KK:H}
In $\A$, we have
$$
A_w \lambda = (w \cdot \lambda) A_w + \sum_{v = wr_\alpha \lessdot w} \ip{\alpha^\vee}{\lambda} A_v.
$$
\end{lem}

Let $\phi_0: S \to \Z$ be defined by evaluation at 0.  
Let $\phi_0:\A\to\A_0$ be the map defined by $\phi_0(\sum_w a_w A_w)
= \sum_w \phi_0(a_w) A_w$ for $a_w\in S$. Let $\B=Z_\A(S)$ be the
Peterson subalgebra \cite{P}, the centralizer subalgebra of $S$ in
$\A$ and let $$\B_0 = \{b\in \B\mid \phi_0(b s)=\phi_0(s) b\text{
for all $s\in S$} \}$$ be the Fomin-Stanley subalgebra~\cite{Lam:Schub}.

\begin{thm} \label{T:H} \
\begin{enumerate}
\item \cite{P}
For each $w\in \Wz$ there is a unique element $\jh_w\in \B$ such that
\begin{eqnarray*}
  \jh_w &\in& A_w + \bigoplus_{v\in W\setminus \Wz} S A_v. \\
  \B &=& \bigoplus_{w\in \Wz} S \, \jh_w.
\end{eqnarray*}
\item \cite{Lam:Schub}
\begin{equation*}
  \B_0 = \phi_0(\B).
\end{equation*}
\item For each $w\in \Wz$, $\phi_0(\jh_w)$ is the unique element of $\B_0$ such that
\begin{eqnarray*}
  \phi_0(\jh_w) &\in& A_w + \bigoplus_{v\in W\setminus \Wz} \Z
  A_v.\\
  \B_0 &=& \bigoplus_{w\in \Wz} \Z \, \phi_0(\jh_w).
\end{eqnarray*}
\end{enumerate}
\end{thm}
Compare these results with the $K$-theoretic analogues of Theorems~\ref{T:LL}
and~\ref{T:L0}.

For $G=SL_{k+1}$ the element $\phi_0(\jh_w)$ is called a
noncommutative $k$-Schur function \cite{Lam:Schub}.

\subsection{Comparison with the fixed point functions of \cite{KK:K}}
\label{SS:KK}
\subsubsection{M\"obius inversion for Bruhat order}
The M\"obius function for the Bruhat order on $W$ is \cite{D}
$$(v,w) \mapsto (-1)^{\ell(w)-\ell(v)}\chi(v\le w)$$
where $\chi(P)=1$ if $P$ is true and $\chi(P)=0$ if $P$ is false.
In other words, let $M$ be the $W\times W$ incidence matrix $M_{vw} = \chi(v\le w)$ of the Bruhat order,
and $N$ the M\"obius matrix $N_{vw} = \chi(v\le w) (-1)^{\ell(w)-\ell(v)}$. Then
$M$ and $N$ are inverse:
\begin{equation} \label{E:mobius}
  \sum_{\substack{v \\ u\le v\le w}} (-1)^{\ell(w)-\ell(v)} = \delta_{uw} =   \sum_{\substack{v\\ u\le v\le w}} (-1)^{\ell(v)-\ell(u)}.
\end{equation}

\subsubsection{Kostant and Kumar functions}
We now return to the ($K$-theoretic) notations of Section \ref{S:KK}. 
The following Lemma is standard.
\begin{lem}
\begin{equation} \label{E:DtoT}
  y_w = \sum_{v\le w} T_v.
\end{equation}
\end{lem}

For $v\in W$ define $\pKK^v\in \F(W,Q(T))$\footnote{The functions denoted $\psi^v(w)$ in \cite{KK:K}
are equal to the functions we denote by $\pKK^{v^{-1}}(w^{-1})$.} by
\begin{equation*}
  \pKK^v(y_w)=\delta_{vw}.
\end{equation*}
This is equivalent to
\begin{align} \label{E:psiKKdef}
  w = \sum_v \psi^v_{KK}(w) y_v.
\end{align}
By \eqref{E:psiKKdef} and \eqref{E:DtoT} we have we have
\begin{equation*}
  w = \sum_v \psi^v_{KK}(w) \sum_{u\le v} T_u = \sum_u T_u \sum_{v \ge u} \psi^v_{KK}(w).
\end{equation*}
By Proposition \ref{P:Kloc}, \eqref{E:Tbasis} and \eqref{E:mobius}
we have
\begin{align*}
  \psi^u &= \sum_{v\ge u} \psi^v_{KK}, \\
  \psi^u_{KK} &= \sum_{v\ge u} (-1)^{\ell(v)-\ell(u)} \psi^v.
\end{align*}
Recall the definition of $\eta$ from Remark \ref{R:eta}.  One may
show the following. Let $\rho$ be the sum of fundamental weights.

\begin{lem} For all $v,w\in W$,
\begin{equation*}
  \pKK^v(w) = (-1)^{\ell(v)} e^{\rho-w\rho} \eta(\psi^v(w)).
\end{equation*}
\end{lem}

\begin{rem} \label{R:KKfuncgeom}
Let $\partial X_v = X_v\setminus X_v^o$ be the boundary of the
Schubert cell $X_v^o$ in the Schubert variety $X_v$. Then there is
an exact sequence
\begin{equation*}
  0 \to I_{\partial X_v\subset X_v} \to \OO_{X_v} \to \OO_{\partial X_v} \to 0.
\end{equation*}
Since $[\OO_{X_u}] \mapsto \psi^u$ under the isomorphism
$K^T(X)\to\Psi$, $[I_{\partial X_v\subset X_v}] \mapsto
\psi^v_{KK}$.
\end{rem}

\subsection{Tables}

\subsubsection{Table on Grassmannians versus $k$-bounded partitions}
\label{S:grass-kbounded}
We list the correspondence between reduced words for Grassmannian elements and $k$-bounded
partitions, where $k=n-1$.

\begin{equation*}
\begin{array}[t]{|l|l|l|}
\hline n & \text{$k$-bounded partition} & w\in W_\af^I\\ \hline
2 & 1 & 0\\
& 11 & 10\\
&111 & 010\\
& 1111 & 1010\\
& 11111 & 01010\\ \hline
3 & 1 & 0\\
& 2 & 10\\
& 11 & 20\\
& 21 & 210\\
& 111 & 120\\
& 22 & 0210\\
& 211 & 2120\\
& 1111 & 0120\\
& 221 & 10210\\
& 2111 & 02120\\
& 11111 & 20120\\ \hline
\end{array}
\qquad
\begin{array}[t]{|l|l|l|}
\hline n & \text{$k$-bounded partition} & w\in W_\af^I\\ \hline
4 & 1 & 0\\
& 2 & 10\\
& 11 & 30\\
& 3 & 210\\
& 21 & 130\\
& 111 & 230\\
& 31 & 3210\\
& 22 & 0130\\
& 211 & 2130\\
& 1111 & 1230\\
& 32 & 03210\\
& 311 & 32130\\
& 221 & 20130\\
& 2111 & 21230\\
& 11111& 01230\\ \hline
\end{array}
\end{equation*}

\subsubsection{$\hat{SL}_2$}
Set $\alpha = \alpha_1  = -\alpha_0$.  We have $t_\alpha = r_0 r_1$
and $t_{-\al}=r_1r_0$. Indexing $T_w$ and $k_w$ by reduced words, we
have
\begin{align*}
t_\alpha &= (1-e^{-\alpha})^2 T_{01} + (1-e^{-\alpha})(T_0 + T_1) + 1 \\
t_{-\alpha} &= (1-e^{\alpha})^2 T_{10} + (1-e^{\alpha})(T_0 + T_1) + 1
\end{align*}
\begin{align*}
k_{\emptyset} &= 1 \\
k_0 &= T_0  + T_1 + (1-e^{-\alpha})T_{01} \\
k_{10} &= T_{10} + e^{-\alpha} T_{01}
\end{align*}
So
\begin{align*}
\phi_0(k_{\emptyset}) &= 1 \\
\phi_0(k_0) &= T_0 + T_1 \\
\phi_0(k_{10}) &= T_{10} + T_{01}
\end{align*}
In general,
\begin{equation*}
  \phi_0(k_{\sigma_r}) = T_{\sigma_r} + T_{\sigma_{-r}},
\end{equation*}
where $\sigma_r$ are the elements in~\eqref{E:sl2gens}.

\newpage
\begin{landscape}
\subsubsection{Table of $\phi_0(k_w)$}
\label{S:tables_k} We index $T_w$ by reduced words.

\begin{equation*}
\begin{array}[t]{|l||l|l|}
\hline n & \text{$w$} & \phi_0(k_w) \\ \hline
3& \emptyset & 1 \\
& 0 & T_0 + T_1 + T_2\\
& 10 & T_{10} + T_{02} +  T_{21}\\
& 20 & T_{20}+T_{01}+T_{12}\\
& 210 & T_{210}+T_{020}+T_{021}+T_{101}+T_{102}+T_{212}\\
& 120 & T_{120}+T_{201}+T_{202}+T_{010}+T_{012}+T_{121}\\
& 0210 & T_{0210}+T_{1021}+T_{2102}\\
& 1210& T_{1210}+T_{0201}+T_{0212}+T_{1020}+T_{1012}+T_{2101}-T_{020}-T_{101}-T_{212}\\
& 0120 & T_{0120}+T_{2012}+T_{1201}\\
\hline
4 & \emptyset & 1 \\
& 0 & T_0 + T_1 + T_2 + T_3\\
& 10 & T_{10} + T_{21} + T_{32} +T_{03} + T_{02} + T_{13}\\
& 30 & T_{30}+T_{01}+T_{12}+T_{23}+T_{02}+T_{13}\\
& 210 & T_{210}+T_{103}+T_{032}+T_{321}\\
& 310 & T_{130}+T_{132}+T_{021}+T_{023}+T_{030}+T_{031}+T_{320}+T_{323}+T_{213}+T_{212}+T_{101}+T_{102}-T_{02}-T_{13}\\
& 230 & T_{230}+T_{301}+T_{012}+T_{123}\\
& 3210 & T_{3210}+T_{3212}+T_{3213}+T_{2101}+T_{2102}+T_{2103}+T_{1030}+T_{1031}+T_{1032}+T_{0320}+T_{0321}+T_{0323} \\
& 0310 & T_{0310}+T_{0213}+T_{1302}+T_{1021}+T_{2132}+T_{3203}\\
& 2310 & T_{2310}+T_{3201}+T_{3230}+T_{3231}+T_{0301}+T_{0302}+T_{0312}+T_{0210}+T_{0212}+T_{0230}+T_{0231}+T_{0232}\\
& & +T_{1301}+T_{1303}+T_{1320}+T_{1321}+T_{1323}+T_{1012}+T_{1013}+T_{1023}+T_{2120}+T_{2123}\\
& 1230 & T_{1230}+T_{1232}+T_{1231}+T_{2303}+T_{2302}+T_{2301}+T_{3010}+T_{3013}+T_{3012}+T_{0120}+T_{0123}+T_{0121}\\
\hline
\end{array}
\end{equation*}
\end{landscape}

\newpage
\begin{landscape}
\subsubsection{Table of $g^{(k)}_\la$}
\label{S:tables g}
We index $g_\la^{(k)}$ by $k$-bounded partitions.
\begin{equation*}
\scalebox{.8}{
\begin{array}[t]{|l||l|l|l|}
\hline n & \la & \text{$g^{(n-1)}_\la$ in terms of $s_\la$} & \text{$g_w$ in terms of $s^{(n-1)}_\la$}\\ \hline
2
&1 & s_{1} & s^{(1)}_{1}\\
&11 & s_{1} + s_{1 1} + s_{2} & s^{(1)}_{1} + s^{(1)}_{11} \\
&111 & s_{1} + 2 s_{1 1} + s_{1 1 1} + 2 s_{2} + s_{3} + 2 s_{2 1}
           & s^{(1)}_{1} + 2 s^{(1)}_{11} + s^{(1)}_{111}\\
&1111 & s_{1} + 3 s_{1 1} + 3 s_{1 1 1} + s_{1 1 1 1} + 3 s_{2} +
                 3 s_{3} + 6 s_{2 1} + 3 s_{3 1} + 2 s_{2 2} + 3 s_{2 1 1} + s_{4}
             & s^{(1)}_{1} + 3 s^{(1)}_{11} + 3 s^{(1)}_{111} + s^{(1)}_{1111} \\
&11111 & s_{1} + 4 s_{1 1} + 6 s_{1 1 1} + 4 s_{1 1 1 1} +
   s_{1 1 1 1 1} + 4 s_{2} + 6 s_{3} + 12 s_{2 1} + 12 s_{3 1} + 8 s_{2 2}
    & s^{(1)}_{1} + 4 s^{(1)}_{11}+ 6 s^{(1)}_{111}+ 4 s^{(1)}_{1111} + s^{(1)}_{11111}\\
    &&+ 12 s_{2 1 1} + 4 s_{4} + s_{5} + 4 s_{4 1} + 6 s_{3 1 1} + 5 s_{2 2 1} + 4 s_{2 1 1 1} + 5 s_{3 2} &
\\ \hline
3
&1 & s_{1} & s^{(2)}_{1}\\
&2 & s_{2} & s^{(2)}_{2}\\\
&11 & s_{1} + s_{11} &   s^{(2)}_{1} + s^{(2)}_{11} \\
&21 & s_{2} + s_{2 1} + s_{3} & s^{(2)}_{2} + s^{(2)}_{21}\\
&111 & s_{1} + s_{2} + 2 s_{1 1} + s_{2 1} + s_{1 1 1}
           & s^{(2)}_{1} + 2 s^{(2)}_{11} + s^{(2)}_{111}+s^{(2)}_{2}\\
&22 & s_{2} + s_{2 1} + s_{2 2} + s_{3} + s_{4} + s_{3 1}
          & s^{(2)}_{2} + s^{(2)}_{21} + s^{(2)}_{22} \\
&211 & s_{2 1} + s_{2 1 1} + s_{3} + s_{3 1}
          & s^{(2)}_{21} + s^{(2)}_{211}\\
&1111 & s_{1} + 2 s_{2} + 3 s_{1 1} + 3 s_{2 1} + 3 s_{1 1 1} + s_{2 2} + s_{2 1 1} + s_{1 1 1 1}
         &   s^{(2)}_{1} + 3 s^{(2)}_{11} + 3 s^{(2)}_{111} + s^{(2)}_{1111} + 2 s^{(2)}_2 \\
&221 & s_{2} + 2 s_{2 1} + 2 s_{2 2} + s_{2 1 1} + s_{2 2 1} + 2 s_{3} + 2 s_{4} + 3 s_{3 1}
        &  s^{(2)}_2 + 2 s^{(2)}_{21} + s^{(2)}_{211} + 2 s^{(2)}_{22} + s^{(2)}_{221}\\
        && + s_{3 1 1} + 2 s_{3 2} + s_{5} + 2 s_{4 1} &\\
&2111 & 2 s_{2 1} + s_{2 2} + 3 s_{2 1 1} + s_{2 2 1} + s_{2 1 1 1} + 2 s_{3} + s_{4} + 4 s_{3 1}
        & 2 s^{(2)}_{21} + 3 s^{(2)}_{211} + s^{(2)}_{2111} + s^{(2)}_{22}\\
        && + 2 s_{3 1 1} + s_{3 2} + s_{4 1} &\\
&11111 & s_{1} + 3 s_{2} + 4 s_{1 1} + 8 s_{2 1} + 6 s_{1 1 1} + 4 s_{2 2} + 7 s_{2 1 1} + 4 s_{1 1 1 1}
        & s^{(2)}_1 + 4 s^{(2)}_{11} + 6 s^{(2)}_{111} + 4 s^{(2)}_{1111} + s^{(2)}_{11111} + 3 s^{(2)}_2\\
        && + 2 s_{2 2 1} + 2 s_{2 1 1 1} + s_{1 1 1 1 1} + 2 s_{3} + 3 s_{3 1} + s_{3 1 1} +s_{3 2}
        &   + 2 s^{(2)}_{21} + 3 s^{(2)}_{211}
\\ \hline
4
&1 & s_{1} & s^{(3)}_1\\
&2 & s_2 & s^{(3)}_2\\
&11 & s_1+s_{11} & s^{(3)}_1+s^{(3)}_{11}\\
&3 & s_3 & s^{(3)}_3\\
&21 & s_2+s_{21} & s^{(3)}_2 + s^{(3)}_{21}\\
&111 &  s_1 + 2s_{11} + s_{111} & s^{(3)}_3 + 2s^{(3)}_{11} + s^{(3)}_{111}\\
&31 & s_3 + s_{31} + s_4 & s^{(3)}_3+s^{(3)}_{31}\\
&22 & s_2+s_{21}+s_{22} & s^{(3)}_2 + s^{(3)}_{21} + s^{(3)}_{22}\\
&211 & s_2 + 2s_{21}+ s_{211} + s_{3}+ s_{31} & s^{(3)}_2 + 2 s^{(3)}_{21} + s^{(3)}_{211}
     + s^{(3)}_3\\
&1111 & s_1+ 3s_{11}+ 3s_{111} + s_{1111} + s_2 + 2s_{21} + s_{211} &
    s^{(3)}_1 + 3 s^{(3)}_{11} + 3 s^{(3)}_{111} + s^{(3)}_{1111} + s^{(3)}_2 + 2 s^{(3)}_{21}
\\ \hline
\end{array}
}
\end{equation*}
\end{landscape}

\newpage

\begin{landscape}
\subsubsection{Table for the coproduct of $g^{(k)}_\la$}
\label{S:coproduct g}

The following table gives $\Delta(g^{(k)}_\la) = \sum_{\nu,\mu} c_\la^{\nu\mu}
g_\nu^{(k)}\otimes g_\mu^{(k)}$,
where we suppress the superscript $(k)$ in the table and write $g_\nu\otimes g_\mu$ instead of
$g_\nu^{(k)}\otimes g_\mu^{(k)}$ with $\nu$ and $\mu$ being $k$-bounded partitions.
\begin{equation*}
\begin{array}[t]{|l|l|l|}
\hline n & \la & \Delta(g^{(n-1)}_\la) \\ \hline
2
&1 & g_{1} \otimes g_\emptyset + g_\emptyset \otimes g_1\\
&11 & g_{11} \otimes g_\emptyset +2 g_1 \otimes g_1 +
     g_\emptyset \otimes g_{11}\\
&111 & g_{111} \otimes g_\emptyset + 3 g_{11} \otimes g_1
   +3 g_1 \otimes g_{11} +  g_\emptyset \otimes g_{111}
   -2 g_1 \otimes g_1\\
&1111& g_{1111} \otimes g_\emptyset + 4 g_{111} \otimes g_1
    +6 g_{11} \otimes g_{11} + 4 g_{1} \otimes g_{111}
    + g_\emptyset \otimes g_{1111} - 5 g_{11} \otimes g_1
    -5 g_1 \otimes g_{11} + 2 g_1 \otimes g_1\\
&11111 & g_{11111} \otimes g_\emptyset + 5 g_{1111} \otimes g_1
    +10 g_{111} \otimes g_{11} + 10 g_{11} \otimes g_{111}
    +5 g_1 \otimes g_{1111} + g_{\emptyset} \otimes g_{11111}\\
&&    -9 g_{111} \otimes g_1 - 16 g_{11} \otimes g_{11}
    -9 g_1 \otimes g_{111} + 7 g_{11} \otimes g_1
    +7 g_1 \otimes g_{11} -2 g_1 \otimes g_1\\ \hline
3
&1& g_1 \otimes g_\emptyset + g_\emptyset \otimes g_1 \\
&22 & g_2 \otimes g_\emptyset + g_1 \otimes g_1 + g_\emptyset \otimes g_2\\
&11 & g_{11} \otimes g_\emptyset + g_1 \otimes g_1 + g_\emptyset \otimes g_{11}\\
&21 & g_{21} \otimes g_\emptyset + g_{11} \otimes g_1 +2g_2 \otimes g_1
    + g_1\otimes g_{11} + 2g_1\otimes g_2 + g_\emptyset \otimes g_{21}
    - g_1 \otimes g_1\\ \
&111& g_{111} \otimes g_\emptyset + 2g_{11} \otimes g_1 + g_2\otimes g_1
    +2 g_1 \otimes g_{11} + g_\emptyset \otimes g_{111} + g_1 \otimes g_2
    - g_1 \otimes g_1\\
&22 & g_{22} \otimes g_\emptyset +2 g_{21} \otimes g_1 + g_{11}\otimes g_{11}
    + g_2\otimes g_{11} + g_{11} \otimes g_2 + 3 g_2 \otimes g_2
    +2 g_1\otimes g_{21} + g_\emptyset \otimes g_{22}
    - g_2 \otimes g_1 - g_1\otimes g_2\\
&211& g_{211}\otimes g_\emptyset + g_{111}\otimes g_1+g_{21}\otimes g_1
    + g_{11}\otimes g_{11} +2g_2\otimes g_{11} + g_1\otimes g_{111}
    +2g_{11}\otimes g_2 + g_2\otimes g_2 + g_1\otimes g_{21}
    +g_\emptyset\otimes g_{211}\\
    && -2 g_{11}\otimes g_1 -2 g_2\otimes g_1 -2 g_1\otimes g_{11}
    -2g_1\otimes g_2 + g_1\otimes g_1\\
&1111& g_{1111}\otimes g_\emptyset + 2 g_{111}\otimes g_1
    +3g_{11}\otimes g_{11} + g_2\otimes g_{11} + 2g_1\otimes g_{111}
    +g_\emptyset\otimes g_{1111}+g_{11}\otimes g_2 + g_2\otimes g_2
    -g_{11}\otimes g_1 - g_1\otimes g_{11}\\ \hline
4
&1&g_1\otimes g_\emptyset + g_\emptyset\otimes g_1\\
&2& g_2\otimes g_\emptyset + g_1\otimes g_1 + g_\emptyset\otimes g_2\\
&11& g_{11}\otimes g_\emptyset +g_1\otimes g_1 + g_\emptyset\otimes g_{11}\\
&3& g_3\otimes g_\emptyset + g_2\otimes g_1 + g_2\otimes g_1
    +g_\emptyset \otimes g_3\\
&21& g_{21}\otimes g_\emptyset + g_{11}\otimes g_1 +g_2\otimes g_1
    +g_1\otimes g_{11} + g_1\otimes g_2 + g_\emptyset\otimes g_{21}
    -g_1\otimes g_1 \\
&111& g_{111}\otimes g_\emptyset + g_{11}\otimes g_1 + g_1\otimes g_{11}
    +g_\emptyset \otimes g_{111}\\
&31& g_{31}\otimes g_\emptyset + g_{21}\otimes g_1 + 2g_3\otimes g_1
    +g_2\otimes g_{11} + g_{11}\otimes g_2 + 2g_2\otimes g_2
    +g_1\otimes g_{21} + 2 g_1\otimes g_3 + g_\emptyset\otimes g_{31}\\
    &&-g_2\otimes g_1 - g_1\otimes g_2\\
&22& g_{22}\otimes g_\emptyset + g_{21}\otimes g_1 + g_{11}\otimes g_{11}
    +g_2\otimes g_2 + g_1 \otimes g_{21} + g_\emptyset\otimes g_{22}\\
&211& g_{211}\otimes g_\emptyset + g_{111}\otimes g_1+2g_{21}\otimes g_1
    +g_3\otimes g_1+g_{11}\otimes g_{11} + 2g_2\otimes g_{11}
    +g_1\otimes g_{111} + 2g_{11}\otimes g_2 + g_2\otimes g_2\\
    &&+2g_1\otimes g_{21} + g_\emptyset\otimes g_{211} +g_1\otimes g_3
    -g_{11}\otimes g_1 -g_2\otimes g_1 - g_1\otimes g_2 -g_1\otimes g_{11}\\
&1111& g_{1111}\otimes g_\emptyset + 2g_{111}\otimes g_1
    +g_{21}\otimes g_1 +2g_{11}\otimes g_{11} + g_2\otimes g_{11}
    +2g_1\otimes g_{111} + g_\emptyset\otimes g_{1111} + g_{11}\otimes g_2\\
    &&+g_1\otimes g_{21} - g_{11}\otimes g_1 - g_1\otimes g_{11}\\ \hline
 \end{array}
\end{equation*}
\end{landscape}

\newpage

\begin{landscape}
\subsubsection{Table of $G^{(k)}_\la$}
\label{S:table G}
We have suppressed the superscript $(k)$ on $F_\mu^{(k)}$ in the following table.
\begin{equation*}
\scalebox{0.9}{
\begin{array}[t]{|l||l|l|l|}
\hline n & \la & \text{$G^{(n-1)}_\la$ in terms of $s_\la$} & \text{$G^{(n-1)}_\la$ in terms of
$F^{(n-1)}_\la$} \\ \hline
2
&1 &  s_1 - s_{1^2} + s_{1^3} - s_{1^4} + s_{ 1^5} -  s_{1^6} \pm \cdots
    &F_1 - F_{1^2} + F_{1^3}-F_{1^4}+F_{1^5} - F_{1^6} \pm \cdots\\
&11 & s_{1^2} -2 s_{1^3} + 3 s_{1^4} - 4 s_{1^5} + 5 s_{1^6} - 6 s_{1^7} +\pm \cdots
    &F_{1^2}-2F_{1^3}+3F_{1^4}-4F_{1^5} +5F_{1^6} - 6F_{1^7} \pm \cdots\\
&111 &  s_{1^3} - 3 s_{1^4} + 6 s_{1^5} - 10 s_{1^6} + 15 s_{1^7} - 21 s_{1^8}
       \pm \cdots
       & F_{1^3} - 3 F_{1^4} + 6 F_{1^5} - 10 F_{1^6} + 15 F_{1^7} - 21 F_{1^8} \pm \cdots\\
&1111 &  s_{1^4} - 4 s_{1^5} + 10 s_{1^6} - 20 s_{1^7} + 35 s_{1^8} -56 s_{1^9} \pm \cdots
    & F_{1^4} - 4 F_{1^5} + 10 F_{1^6} - 20 F_{1^7} + 35 F_{1^8} -56 F_{1^9} \pm \cdots\\
&11111 & s_{1^5} - 5 s_{1^6}  + 15 s_{1^7} -35 s_{1^8} + 70 s_{1^9} -126 s_{1^{10}} \pm \cdots
    & F_{1^5} - 5 F_{1^6}  + 15 F_{1^7} -35 F_{1^8} + 70 F_{1^9} -126 F_{1^{10}} \pm \cdots \\
\hline
3
&1 &  s_1 - s_{1^2} + s_{1^3} - s_{1^4} + s_{1^5} - s_{1^6} \pm \cdots
    & F_1 - F_{1^2} + F_{1^3} - F_{1^4} + F_{1^5} - F_{1^6} \pm \cdots\\
&2 & s_2 - s_{21} + s_{211} - s_{2111} + s_{21111} - s_{211111} \pm \cdots
    & F_2-F_{21}-F_{111}+F_{211}+F_{1^4}-F_{21^3}-2F_{1^5} +F_{21^4}+2F_{1^6}\pm \cdots\\
&11 & s_{1^2} - 2 s_{1^3} + 3 s_{1^4} - 4 s_{1^5} + 5 s_{1^6} - 6 s_{1^7} \pm \cdots
    & F_{1^2} - 2 F_{1^3} + 3 F_{1^4} - 4 F_{1^5} + 5 F_{1^6} - 6 F_{1^7} \pm \cdots\\
&21 & - s_{1^3} + 2 s_{1^4} - 3 s_{1^5} + s_{21} - s_{2^2} - s_{21^2} + s_{2^21} + s_{21^3} +
   4 s_{1^6}
   &F_{21}-F_{22}-F_{211}+F_{221}+2F_{21^3}+F_{1^5}\\
   &&- s_{2^21^2}- s_{21^4} - 5 s_{1^7} + s_{21^5} + s_{2^2 1^3} - s_{21^6}
   + 6 s_{1^8} - s_{2^21^4} \pm \cdots
   &-F_{2211}-2F_{21^4}-F_{1^6}+F_{221^3}+3F_{21^5}+3F_{1^7}\pm \cdots \\
&111 & s_{1^3}- 3 s_{1^4} + 6 s_{1^5} - 10 s_{1^6} + 15 s_{1^7} - 21 s_{1^8} \pm \cdots
    & F_{1^3} -3F_{1^4} + 6 F_{1^5} - 10 F_{1^6} + 15 F_{1^7} -21 F_{1^8} \pm \cdots \\
&22 & - s_{1^4} + 2 s_{1^5} + s_{2^2} - 2 s_{2^21} + s_{21^3} -3 s_{1^6} + s_{2^3}
   2 s_{2^2+1^2} - 2 s_{21^4}
   &F_{22}-2F_{221}-F_{21^3}+F_{2^3}+2F_{2211}+F_{21^4}\\
  && + 4 s_{1^7} + 3 s_{21^5}- s_{2^31} - 2 s_{2^21^3} -
   4 s_{21^6} - 5 s_{1^8} + 2 s_{2^21^4} + s_{2^31^2}
   &-F_{2^31}-3F_{221^3}-3F_{21^5}-F_{1^7} \pm \cdots\\
   &&- 2 s_{2^2 1^5} + 5 s_{21^7} + 6 s_{1^9} - s_{2^3 1^3} \pm \cdots
   &\\
&211 & - s_{1^4} + 3 s_{1^5} + s_{21^2} -
   s_{2^21} - 2 s_{21^3} - 6 s_{1^6} + s_{2^3} +
   2 s_{2^21^2} + 3 s_{21^4}
   &F_{211}-F_{221}-3F_{21^3}-3F_{1^5}+F_{2^3}+2F_{2211}+6F_{21^4}\\
   && + 10 s_{1^7} - 4 s_{21^5} - 2 s_{2^3 1} - 3 s_{2^2 1^3}
   + 5 s_{21^6} - 15 s_{1^8} + 4 s_{2^2 1^4} + s_{2^4}
   &+7F_{1^6}-2F_{2^31}-5F_{221^3}-14F_{21^5}-25F_{1^7}\pm \cdots\\
   &&+ 3 s_{2^31^2} - 5 s_{2^21^5} - 6 s_{21^7} + 21 s_{1^9} - 4 s_{2^3 1^3} -2 s_{2^4 1} \pm \cdots
   &\\
&1111 & s_{1^4} - 4 s_{1^5} +10 s_{1^6} - 20 s_{1^7} + 35 s_{1^8} - 56 s_{1^9} \pm \cdots
   &F_{1^4} - 4 F_{1^5} +10 F_{1^6} - 20 F_{1^7} + 35 F_{1^8} - 56 F_{1^9} \pm \cdots \\
&221 & s_{2^21} - s_{21^3} - s_{1^6} - 2 s_{2^3} - s_{2^21^2} + 3 s_{21^4} +
   3 s_{1^7} - 5 s_{21^5} + 3 s_{2^31}
   &F_{221}-F_{2^3}-F_{2211}+3F_{2^31}+3F_{221^3}+F_{21^5}\pm \cdots\\
   && + 7 s_{21^6} - 6 s_{1^8} + s_{2^2 1^4} - s_{2^4} - 3 s_{2^31^2}
   - 2 s_{2^2 1^5}- 9 s_{21^7} + 10 s_{1^9}
   &\\
   && + 3 s_{2^31^3} + s_{2^41} - 15 s_{1^{10}} - s_{2^41^2} - 3 s_{2^31^4} +
   3 s_{2^2 1^6} + 11 s_{21^8} \pm \cdots
   &\\
&2111 & - 2 s_{1^5} + s_{21^3} + 7 s_{1^6} - s_{2^21^2} - 2 s_{21^4} -
   16 s_{1^7} + 3 s_{21^5} + s_{2^31}
   &F_{21^3}-F_{2211}-3F_{21^4}+F_{2^31}+3F_{221^3}\\
  &&+ 2 s_{2^21^3} - 4 s_{21^6} + 30 s_{1^8} - 3 s_{2^2 1^4} - s_{2^4} - 2 s_{2^3 1^2}
  + 4 s_{2^2 1^5}
  &+9F_{21^5}+7F_{1^7}\pm \cdots\\
  && + 5 s_{2 1^7} - 50 s_{1^9} + 3 s_{2^3 1^3} + 2 s_{2^41} + 77 s_{1^{10}} - s_{2^5} -
   3 s_{2^4 1^2}
   &\\
   && - 4 s_{2^3 1^4} - 5 s_{2^2 1^6} - 6 s_{2 1^8} \pm \cdots
   &\\
 &11111 &
   s_{1^5} - 5 s_{1^6} + 15 s_{1^7} - 35 s_{1^8} + 70 s_{1^9} - 126 s_{1^{10}} \pm \cdots
   & F_{1^5} - 5 F_{1^6} + 15 F_{1^7} - 35 F_{1^8}\pm\cdots\\
\hline
\end{array}
}
\end{equation*}
\end{landscape}

\end{document}